\documentclass{article}
\usepackage{mathrsfs}
\usepackage{amsmath,amscd,amsthm}
\usepackage{amsmath}

\usepackage{amsthm}%, amscd}

\input{amssym.tex} 

\newtheorem{Theorem}{Theorem}[section]%[\theTheorem.\Alph{Theorem}]

\newtheorem{Proposition}[Theorem]{Proposition}%[section]
\newtheorem{Lemma}[Theorem]{Lemma}%[section]
\newtheorem{Sublemma}[Theorem]{Sublemma}%[section]

\newtheorem{Remark}{Remark}[section]
\newtheorem*{Comments}{Comments}
\newtheorem{Example}{Example}[section]

\newtheorem*{TheoremA}{Theorem A}%[TheoremAA]%[\theTheorem.\Alph{Theorem}]
\newtheorem*{TheoremB}{Theorem B}
\newtheorem*{TheoremC}{Theorem C}
\newtheorem*{TheoremD}{Theorem D}
\newtheorem*{TheoremE}{Theorem E}
\newtheorem*{Theorems}{Theorem}
\newtheorem*{Lemmas}{Lemma}
\newtheorem*{AssumptionB}{Assumption B}
\newtheorem*{AssumptionT}{Assumption T}
\newtheorem*{AssumptionS}{Assumption S}
\newtheorem*{AssumptionT'}{Assumption T${}'$}
\newtheorem*{AssumptionT''}{Assumption T${}''$}

\def\eqref#1{(\ref{#1})}

\def\Aut{\mathop{\hbox{\rm Aut}}}

\def\Cov{\mathop{\hbox{\rm Cov}}}

\def\diam{\mathop{\hbox{\rm diam}}}
\def\intset{\mathop{\hbox{\rm int}}}

\def\dim{\mathop{\hbox{\rm dim}}}
\def\det{\mathop{\hbox{\rm det}}}
\def\id{\mathop{\hbox{\rm id}}}
\def\mod{\mathop{\hbox{\rm mod}}}

\def\supp{\mathop{\hbox{\rm supp}}}
\def\Esup{\mathop{\hbox{\rm Esup}}}
\def\Einf{\mathop{\hbox{\rm Einf}}}
\def\osc{\mathop{\hbox{\rm osc}}}
\def\disp{\displaystyle}
\def\<<{\prec}

\def\leb{\mbox{Leb}}
\def\A{{\mathcal A}}
\def\B{{\mathcal B}}
\def\I{{\mathcal I}}
\def\L{{\mathcal L}}
\def\P{{\mathscr P}}
\def\Q{{\mathcal Q}}
\def\H{{\mathcal H}}

\def\hn{{\hat{\nu}}}
\def\wh X{{\hat{h}}}
\def\hmu{{\hat{\mu}}}

\def\a{\alpha}
\def\b{\beta}
\def\c{\gamma}   \def\C{\Gamma}
\def\d{\delta}
\def\D{\Delta}
\def\e{\varepsilon}

\def\l{\lambda}
   
\def\s{\sigma}   

\def\h{\hat{h}}

\def\1{\mathbbold{1}}
\def\mathbbold{}
\def\mathbb{\Bbb}

\def\wt#1{{\widetilde{#1}}}
\def\wh#1{{\widehat{#1}}}

\def\dd{\zeta} \def\m{\nu}

\begin{document}
\bibliographystyle{plain}

\title {Lower Bounds for the Decay of Correlations in  Non-uniformly Expanding  Maps}
\author{Huyi Hu \thanks{Mathematics Department, Michigan State University,
East Lansing, MI 48824, USA.
e-mail: $<$hu@math.msu.edu$>$.} \and Sandro Vaienti
\thanks{Aix Marseille Universit\'e, CNRS, CPT, UMR 7332, Marseille, France and and Universit\'e
de Toulon, CNRS, CPT, UMR 7332, 83957 La Garde. e-mail:  $<$vaienti@cpt.univ-mrs.fr$>$. HH was supported by  Aix-Marseille University and the University of Toulon during his visits at the Center of Theoretical Physics in Luminy.  SV was supported by the ANR-Project {\em Perturbations}, by the  CNRS-PEPS {\em Mathematical Methods of Climate Theory} and by the {\sc PICS} ( Projet International de Coop\'eration Scientifique), {\em Propri\'et\'es statistiques des syst\`emes dynamiques det\'erministes et al\'eatoires}, with the University of Houston,  n. PICS05968. Part of this work was done while he was visiting the {\em Centro de Modelamiento Matem\'{a}tico, UMI2807}, in Santiago de Chile with a CNRS support (d\'el\'egation).}}

%\date{}

\maketitle

\begin{abstract}
We give conditions under which nonuniformly expanding maps exhibit lower bounds of polynomial type for the decay of correlations and for  a large class of observables.
We show that if the Lasota-Yorke type inequality for the transfer operator
of a first return map are satisfied in a Banach space $\mathcal B$,
and the absolutely continuous invariant measure
obtained is weak mixing, in terms of aperiodicity,
then under some renewal condition,
the maps have polynomial decay of correlations for observables in $\mathcal B. $
We also provide some general conditions that give aperiodicity
for expanding maps in higher dimensional spaces.
As applications, we obtain  lower bounds for  piecewise expanding
maps  with an indifferent fixed point and for which we also allow non-Markov structure and unbounded distortion.
The observables are functions that have bounded variation or satisfy
quasi-H\"older conditions  and have their support  bounded away from  the neutral fixed points.
\end{abstract}
%\begin{keywords}nonuniformly expanding maps; polynomial decay of correlations;
%aperiodicity; Lasota-Yorke type inequalities; bounded variation functions;
%quasi-H\"older functions
%\end{keywords}

%\tableofcontents
%\section{Introduction}
\section{Introduction}
\setcounter{equation}{0}

The  purpose of this paper is to study polynomial decay of correlations
for invariant measures which are  absolutely continuous
with respect to the Lebesgue measure on compact subsets of $\mathbb{R}^n.$ Typically the maps $T$ which we consider are non uniformly expanding and  may neither have
a Markov partition nor exhibit bounded distortion.
The main tool we use is the transfer (Perron-Frobenius) operator on induced subsystems endowed with the first return map.
%Let us call $||R_n||$ a suitable norm (see below) of the $n$-th power of the transfer operator restricted to the level sets with first return time $\tau=n$.
%We will show that if Lasota-Yorke inequalities can be verified for
%the transfer operator of the first return maps,
%and if  $\|R_n\|$ converges to $0$
%at a speed $1/n^{\b+1}$, $\b>1$,
%then the decay rates are given by the sum $\sum_{k=n+1}^\infty \mu(\tau>k)$
%of measure of the sets.
%In the second part of the paper we apply the results
%to piecewise expanding maps with an indifferent fixed point
%in one dimensional and higher dimensional spaces to get
%polynomial decay of correlations.
%The results for maps in higher dimensional spaces with $Df_p=\id$ at
%the indifferent fixed point $p$ are new, and in all the cases,
%the observables are more general than H\"older functions.\\

We now explain in  detail the content of this paper.
Let us consider a non uniformly expanding map $T$ defined on a compact subset
$X\subset \mathbb{R}^n$, with or without discontinuities.
Since we do not have necessarily bounded distortion or Markov partitions,
the H\"older property is  not preserved under the transfer
operator.
Therefore  we will work on Banach spaces $\B$ embedded in
$L^1$ with respect to the Lebesgue measure,
and we will give some conditions on $\B$ under which the results apply, see Assumption B.

Let us now take a subset $\hat{X}\subset X$ and define the first return map $\wh T$.
The \emph{first ingredient} of our theorem is
the \emph{Lasota-Yorke inequality} for the transfer operator $\wh \P$
of $\wh T$ with respect to the norms $\|\cdot\|_\B$ and $\|\cdot\|_{L^1}$.
Hence, $\wh \P$ has a fixed point $\hat h$ that defines an absolutely
continuous measure $\hat \mu$ invariant under $\wh T$.
The measure $\hat \mu$ can be extended to a measure $\mu$ on $X$ invariant
under $T$.  We may assume ergodicity for $\hat \mu$, otherwise we
take an ergodic component.  Then the ergodicity of $\hat \mu$ gives
ergodicity of $\mu$.  However, we also need some mixing property for $\mu$.
Therefore our \emph{second ingredient} is to require that
the function $\tau$ given by the first return time is \emph{aperiodic},
which is equivalent to the weak mixing of $\mu$ for $T$.
The \emph{third ingredient} is precise tail estimates as they are required in the renewal theory approach. In this regard, let us  call $||R_n||$ the operator  norm (see below) of the $n$-th power of the transfer operator restricted to the level sets with first return time $\tau=n$; then we  ask  that $\sum_{k=n+1}^{\infty}||R_k||$   decays  at least as $n^{-\beta}$, with $\beta>1$.
Such a decay  gives also an  estimate, through the exponent $\b$,  of the {\em error term} denoted by the function $F_{\b}(n)$ in the basic inequality (\ref{fThmA}) of   Theorem A below.  Whenever    that error term goes to zero
 faster than $\sum_{k=n+1}^{\infty}\mu(\tau > k)$, the latter sum gives a lower bound for the decay of correlations and we will refer to this situation as the
 optimal rate: this will be shown to hold in the situations of Section 5.
 %In any case it follows
% from the general theory of renewal developed by Sarig \cite{Sr} and successively improved by Gou\"ezel \cite{Go},  that the
%decay of correlations $$\Cov(f, g\circ T^n):=|\int f \ g\circ T^n\ d\mu - \int f d\mu \ \int g \ d\mu|,$$  is polynomial
%for functions $f\in \B$ and $g\in L^\infty$  (with respect to the Lebesgue measure) and
%$\supp f, \supp g\subset \wh X$.
%We would like to point out that we do not need to assume
%the existence of   absolutely continuous invariant measures (acim), since we will get them
 %by the Lasota-Yorke type inequalities (see below), which will  provide us with the usual quasi-compactness picture for the transfer operators on a function space larger than
%H\"older continuous functions.

The proof of aperiodicity in Theorem B is particularly technical.
We use some results of the theory developed in the paper \cite{ADSZ},
where aperiodicity is proved for a large class of interval maps,
and some methods in \cite{AD} for skew product rigidity.
We extend the aperiodicity result to the multidimensional setting
without Markov partitions thus pursuing the program started in \cite{ADSZ}, which was just oriented to treat the non-Markov cases especially for one-dimensional systems.\\

Several examples will be presented and discussed in detail.

In the one-dimensional case we use the set of bounded variation functions
for the Banach space $\B$, and we find that the decay rates are of
 order $n^{\b-1}$ if near the fixed point the map has the form
$T(x)\approx x+x^{1+\c}$, $\c\in(0,1)$ and $\b=1/\c$. Upper bounds for the decay of correlations for these kinds of maps were already given by Young \cite{Yo2} and by Melbourne and Terhesiu, see Section 5.3 in  \cite{MT}.

%As a matter of fact, one of the main goals in our paper is to obtain polynomial decay
%of correlations for piecewise smooth expanding maps with
%an indifferent fixed point in \emph{higher dimensional spaces}.

We then consider a large class of  maps in higher dimensions that we introduced in a previous paper \cite{HV}, and in  sections 4 and 5 we will specify the {\em roles} of the derivative and of the determinant  in order to get a lower bound for the decay of correlations.

In particular we will obtain optimal rates under the assumption
that all the pre-images of some neighborhood of $p$ do not intersect
discontinuities,
(see Theorem~E and examples in Subsection~5.2 for more details). This  is satisfied for instance  whenever $T$ has a Markov partition, even countable,
see Remark~\ref{Rexist}.
Moreover in Example~\ref{example5} and thereinafter we show
the  existence  of these systems with
all the pre-images of some neighborhood of $p$  not intersecting
discontinuities, but  without any Markov structure.

%For piecewise expanding interval maps with indifferent fixed points, it is relatively easy to get the desired spectral properties on the space of bounded variation functions and to estimate
 %decreasing rates for $\|R_n\|$:
%our theorem allows then to  get optimal polynomial decay rates of correlation.
We would like to point out two main issues which make
the higher dimensional case   more complicated. The first  is
 due to {\em unbounded distortion} of the map. This is
caused by different expansion rates in different directions
as a point move away from the indifferent fixed point
even if $DT_p=\id$ at the fixed point (see Example~1, part (A) in \cite{HV}).
The second comes from the difficulty to estimate the decreasing rates of the norm $\|R_n\|$
for quasi-H\"older spaces: Theorems D and E deal with these situations under certain hypotheses. One surely needs
 more work to weaken those assumptions and achieve optimal decay for a much larger class of maps.

\medskip

%and we  give some conditions
%on Banach spaces through which one can obtain some statistical  properties
%such as existence of a physical measure and decay of correlations.

%We would like to remark at this point that the functional space
%is abstract, as long as certain general assumptions
%(Assumption B(d) to (f)) are satisfied.
%In the applications,  we present two type of Banach spaces.
%It seems interesting to find more different spaces to deal with
%different kinds of dynamical systems.

%$\mathscr{O}\mathcal {P}$

\bigskip
\bigskip
%\part{Conditions for Polynomial Decay Rates}
%\noindent{\bf \Large Part I: Conditions for Polynomial Decay Rates}

\section{Assumptions and statements of results}
\setcounter{equation}{0}

Let $X\subset \Bbb R^m$ be a subset with positive Lebesgue measure
$\nu$.
We assume $\nu (X)=1$.
%$\overline{\intset X}=X$ or a compact Riemannian manifold,
%Let $d$ be the (euclidean) metric induced from $\Bbb R^m$.

The transfer  operator
$\P=\P_\nu: L^1(X, \nu)\to L^1(X, \nu)$ is defined by
$\int \psi\circ T \phi d\nu= \int \psi \P\phi d\nu$\
$\forall \phi\in L^1(X, \nu)$, $\psi\in L^\infty(X, \nu)$.
% where the latter $L^p$ spaces are defined with respect to the Lebesgue measure $\nu$ on the Borel $\sigma$-algebra of $X$.

Let ${\wh X}\subset X$ be a measurable subset of $X$
with positive Lebesgue measure.

Recall that the first return map of $T$ with respect to ${\wh X}\subset X$
is defined by $\wh T(x)=T^{\tau(x)}(x)$,
where $\tau(x)=\min\{i\ge 1: T^ix\in {\wh X}\}$
is the return time. We put $\hn$ the normalized Lebesgue measure on $\wh X$.
Then we let ${\wh \P}={\wh \P}_\hn$ be the transfer operator of $\wh T$.

Moreover we define
\begin{eqnarray}~\label{defRnTn}
R_nf=\mathbbold{1}_{\wh X}\cdot \P^n(f\mathbbold{1}_{\{\tau=n\}})
\quad \mbox{and}  \quad
T_nf=\mathbbold{1}_{\wh X}\cdot \P^n(f\mathbbold{1}_{{\wh X}})
\end{eqnarray}
for any function $f$ on ${\wh X}$.
For any $z\in \mathbb C$, denote
$\disp R(z)=\sum_{n=1}^\infty z^nR_n$.
It is clear that ${\widehat \P}=R(1)=\sum_{n=1}^\infty R_n$.

For simplicity of notation, we regard the space $L^1({\wh X}, \hn)$
as a subspace of $L^1(X, \nu)$ consisting of functions supported on ${\wh X}$,
and we denote it by  $L^1(\hn)$ or sometimes by  $L^1$ and when no ambiguity arises. We will denote ${\mathbb D}=\{z\in \mathbb C: |z|< 1\}$ and
${\mathbb S}=\{z\in \mathbb C: |z|= 1\}$.

Suppose that there is a seminorm $|\cdot|_\B$ for functions
in $L^1({\wh X}, \hn)$.
Consider the set $\B=\B({\wh X})=\{f\in L^1({\wh X}, \hn): |f|_\B < \infty\}$.
Define a norm on $\B$ by
$$\|f\|_\B=|f|_\B+\|f\|_1$$
for $f\in \B$, where $\|f\|_1$ is the $L^1$ norm.
%We assume that the seminorm is given in such a way
%that $\B$ is complete.  Hence $\B$ is a Banach space.
We assume that $\B$ satisfies the  requirements stated below; the assumptions (a) to (c) will be necessary to establish the spectral gap of the induced transfer operator, while conditions (d) to (f) will be useful to prove aperiodicity. We first define  a set $U\subset \wh{X}$  to be {\it almost open} with respect to $\hn$
if for $\hn$ almost every point $x\in U$, there is a neighborhood
$V(x)$ such that $\hn(V(x)\setminus U)=0$.

\begin{AssumptionB}\label{AB}
\begin{enumerate}
\item[{\rm (a)}] {\rm (Compactness)}
$\B$ is a Banach space and the inclusion
$\B \hookrightarrow  L^1(\hn)$ is compact;
that is, any bounded closed set in $\B$ is compact in $L^1(\hn)$.

\item[{\rm (b)}] {\rm (Boundedness)}
The inclusion $\B \hookrightarrow  L^\infty(\hn)$ is bounded;
that is, $\exists C_b>0$ such that $\|f\|_\infty\ \le C_b\|f\|_\B$
for any $f\in \B$.

\item[{\rm (c)}] {\rm (Algebra)}
$\B$ is an algebra with the usual sum and product of functions,
in particular there exists a constant $C_a$ such
that $\|fg\|_\B\le C_a\|f\|_\B\|g\|_\B$ for any $f,g\in \B$.
\item[{\rm (d)}] {\rm (Denseness)}
The image of the inclusion $\B\hookrightarrow  L^1(\hn)$
is dense in $L^1(\hn)$.

\item[{\rm (e)}] {\rm (Lower semicontinuity)}
For any sequence $\{f_n\}\subset \B$ with $\disp \lim_{n\to \infty} f_n=f$
$\hn$-almost everywhere, $\disp |f|_\B\le \liminf_{n\to\infty}|f_n|_\B$.

\item[{\rm (f)}] {\rm (Openness)}
For any nonnegative function $f\in \B$, the set $\{f>0\}$ is almost open
with respect to $\hn$.
\end{enumerate}
\end{AssumptionB}
\begin{Remark}
Assumption B(f) means that functions in $\B$ are not far
from continuous functions.
\end{Remark}

The possibility of computing a lower bound for the decay of correlations relies on
a  result first established by  Sarig \cite{Sr} and  improved by Gou\"ezel \cite{Go}. We now state the sufficient conditions on our systems which will allow us to apply those results and we will comment later on about such implication.

%\begin{Theorems}\mbox{{\bf \cite{Sr, Go}}}
%Let $T_n$ be bounded operators on a Banach space $\B$ such that
%$T(z) = I + \sum_{n\ge 1}z^nT_n$ converges in $\Hom(\B, \B)$
%for every $z\in {\mathbb D}$. Assume that:

%\begin{enumerate}
%\item[{\rm (R1)}] {\rm (Renewal equation)}
%for every $z\in {\mathbb D}$,
%$T(z)=(I-R(z))^{-1}$,  where $R(z) =\sum_{n\ge 1}z^nR_n$,
%$R_n\in \Hom(\B, \B)$ and $\sum_{n\ge 1}\|R_n\|< +\infty$.

%\item[{\rm (R2)}] {\rm (Spectral gap)}
%$1$ is a simple isolated eigenvalue of $R(1)$.

%\item[{\rm (R3)}] {\rm (Aperiodicity)}
%for every $z\in {\overline{\mathbb D}}\setminus \{1\}$, $I - R(z)$
%is invertible.
%\end{enumerate}
%Let $P$ be the eigenprojection of $R(1)$ at $1$.
%If $\sum_{k> n}\|R_k\|=O(1/n^\b)$ for some $\b > 1$
%and $P R'(1)P \not=0$, then for all $n$,
%\begin{eqnarray}\label{fGouezel}
%T_n=\frac{1}{\l}P + \frac{1}{\l^2}\sum_{k=n+1}^\infty P_k+ E_n,
%\end{eqnarray}
%where $\l$ is given by $P R'(1)P=\l P$, $P_n=\sum_{k> n}P R_kP$
%and $E_n\in \Hom(\B, \B)$ satisfies
%$\|E_n\| =O(1/n^\b)$ if $\b > 2$, $O(\log n/n^2)$ if $\b=2$,
%and $O(1/n^{2\b-2})$ if $2>\b > 1$.
%\end{Theorems}
%In the dynamical setting we are interested in,  the operators $T_n$ and $R_n$ are defined by \eqref{defRnTn}. In order to check points $(R1)$ to $(R3)$ in the statement above, one needs equivalent dynamical properties, which are summarized in the following spectral-like assumption.
\begin{AssumptionS}\label{SA}
Let $X\subset \Bbb R^m$ be a compact subset
and ${\wh X}\subset X$ be a compact subset of $X$.

Let $T: X\to X$ be a map whose first return map with respect to ${\wh X}$
is $\wh T=T^\tau$,
and $\B$ be a Banach space satisfying Assumption~B(a) to (c).
We assume the following.
\begin{enumerate}
%\item[{\rm (S1)}]  {\rm (Lasota-Yorke inequality)}
%There exist constants $\eta\in (0,1)$ and $D>0$ such that for any
%$f\in \B$,
%\begin{eqnarray}\label{LYineq}
%|{\widehat \P} f|_\B\le \eta|f|_\B+D\|f\|_1;
%\end{eqnarray}

\item[{\rm (S1)}]  {\rm (Quasi compactness)}
There exist constants $B, {\hat D}>0$ and $\hat\eta\in (0,1)$
such that for any
$f\in \B$, $z\in \overline{\mathbb D}$,
\begin{eqnarray}\label{LYineq2}
\|R(z)^nf\|_\B\le |z^n|\bigl(B\hat\eta^n\|f\|_\B+{\hat D}\|f\|_1\bigr).
\end{eqnarray}
Note that for $z=1$ we obtain the usual Lasota-Yorke inequality for the operator $\widehat \P.$
%\item[{\rm (S3)}] {\rm (Ergodicity)}
%The measure $\hmu$ given by $\hmu(f)=\hn(\h f)$ is ergodic,
%where $\h$ is a fixed point of ${\widehat \P}$.

\item[{\rm (S2)}] {\rm (Aperiodicity)}
The function $e^{it\tau}$ given by the return time is aperiodic,
that is, the only solution for $e^{it\tau}=f/f\circ \hat{T}$ which holds
almost everywhere with a measurable function
$f: {\wh X}\to {\mathbb S}$, is provided by  $f$ constant almost everywhere and $t=0$. It will follow that the measure $\hmu$ given by $\hmu(f)=\hn(\h f),$ where  $\h$ is a fixed point of ${\widehat \P},$ is ergodic since aperiodicity is equivalent to weak-mixing (see e.g. \cite{PP}).

\item[{\rm  (S3) }] {\rm (Return times tail)}
The ${\cal B}$-norm of the operator $R_n$ is summable and  satisfies
$\sum_{k=n+1}^\infty \|R_k\|_\B= O(n^{-\b})$ for some $\b>1.$

\end{enumerate}

\end{AssumptionS}

As we said above, a useful reformulation of the theorems in \cite{Sr} and  \cite{Go} allows us to get the following result:
\begin{TheoremA}
Let us suppose that Assumption (S) is satisfied; then there exists a constant $C>0$ such that for any function $f\in \B$,
$g\in L^\infty(X,\nu)$  with $\supp{f}, \ \supp{g}\subset {\wh X}$,
\begin{eqnarray}\label{fThmA}
\Bigl| \Cov(f, g\circ T^n)
-\bigl(\sum_{k=n+1}^\infty \mu(\tau>k)\Bigr)\int fd\mu\int gd\mu\Bigr|
\le C F_\b(n)\|g\|_\infty\|f\|_\B,
\end{eqnarray}
where $F_\b(n)=1/n^\b$ if $\b>2$, $(\log n)/n^2$ if $\b=2$,
and $1/n^{2\b-2}$ if $2>\b>1$.
\end{TheoremA}

\begin{Comments}

\begin{enumerate}
\item  Sarig and Gou\"ezel theory requires that in addition to condition (S3), two more assumptions are satisfied. The first condition asks  that $1$ is a simple isolated eigenvalue of $R(1)$  and this is an immediate consequence of the quasi-compactness of $\wh \P$ and of the ergodicity of $\hmu.$
\item  The second assumption requires   that $1$ is not an
eigenvalue of $R(z)$ for $|z|=1$ with $z\not=1$. Let us  fix
$0<t<2\pi$ and put $z=e^{it};$ if we suppose that $R(z)f=f$ for some nonzero $f\in \B$, by the arguments developed in the proof of the Lemma 6.6 in \cite{Go}, that is equivalent  to the equation $e^{-it\tau}f\circ \wh T=f$ almost everywhere.
By the aperiodicity condition $(S2)$ we conclude that $t=0$
and $f$ is a constant  $\hmu$-almost everywhere  which is a contradiction.

\end{enumerate}
\end{Comments}

Assumption $(S2)$  is usually difficult to check. However, for
piecewise expanding systems, the condition can  be verified and we will
give some sufficient conditions in Theorem B below.
%\end{Remark}

%\begin{Remark}\label{RmkA4}
%If the Banach space $\B$ is also an algebra in the sense that
%$f,g\in \B$ implies $fg\in \B$, then $\F\supset \B$.
%This is the case in the applications we discuss in the next section.
%\end{Remark}

\medskip
The more general version of aperiodicity is  the following. Let
${\mathbb G}$ be a locally compact Abelian polish group. A
measurable function $\phi: {\wh X}\to {\mathbb G}$ is {\it aperiodic} if
the only solutions for $\c\circ \phi =\l f/f\circ T$ almost
everywhere with $\c\in \wh {\mathbb G}$, $|\l|=1$ and a measurable
function $f: {\wh X}\to {\mathbb G}$ are $\c=1$, $\l=1$ and $f$
constant almost everywhere, see \cite{ADSZ} and  references therein.
Here we only consider the case
$\c=\id$, and $\phi=e^{it\tau}$, and ${\mathbb G}$ being the
smallest compact subgroup of ${\mathbb S}$ containing $e^{it}$.
%and therefore Assumption (d) in Theorem A can be verified.

%In the next theorem we provide some conditions for the first
%return map $\wh T$ $T: X\to X$ and its first return map
%$\wh T: {\wh X}\to {\wh X}$ under which the function $e^{it\tau}$
%is aperiodic.

%We say that $T$ has \emph{relatively prime return time} at $x\in \wh X$
%if for any neighborhood $U$ of $x$, there is a point
%$y\in U$ and return times $\tau'(x)$ and $\tau'(y)$
%such that $T^{\tau'(x)}(x),T^{\tau'(x)}(x)\in U$
    %and the greatest common divisor $(\tau'(x), \tau'(y))=1$.
    %Here $\tau'(x)$ and $\tau'(y)$ are not necessary the first return time.

    We denote by $B_\e(\Gamma)$ the $\e$ neighborhood of
    a set $\Gamma\subset X$.
Recall that the notion of {\em almost open} is given before the statement of
Assumption~B.
We now state a few conditions which must be satisfied by all the maps considered from now on.

    \begin{AssumptionT}\label{AT}
    \begin{enumerate}
    \item[{\rm (a)}]  {\rm (Piecewise smoothness)}
    There are countably many disjoint sets $U_1, U_2, \cdots$ almost open
with respect to $\nu$, with     ${\wh X}=\overline{\bigcup_{i=1}^\infty {U_i}}$
a compact set, such that for each $i$,
    $\wh T_i:=\wh T|_{U_i}$ extends to a $C^{1+\a}$ diffeomorphism from
    $\overline U_i$ to its image, and $\tau|_{U_i}$ is constant;
    we will use the symbol $\wh T_i$ to denote the extension as well.

    \item[{\rm (b)}] {\rm (Finite images)} The collection
    $\{\wh TU_i: i=1,2,\cdots\}$ is finite, and
    $\nu (B_\e(\partial \hat{T}U_i))=O(\e)$\ $\forall i=1,2,\cdots$.

    \item[{\rm (c)}] {\rm (Expansion)}
    There exists $s\in (0,1)$ such that
    $d(\wh Tx,\wh Ty)\ge s^{-1}d(x,y)$\ $\forall x,y\in \overline U_i$
    $\forall i\ge 1$.

\item[{\rm (d)}] {\rm (Topological mixing)}
$T: X\to X$ is topological mixing.
%\item[{\rm (d)}] {\rm (Relatively prime return time)}
%$T: X\to X$ has relatively prime return time on a set of positive measure
%with respect to an ergodic measure $\mu$ whose density is given by
%a fixed point $h$ of the transfer operator $\wh \L$.
%{\bf I do not quite like this assumption.  However, it seems the best.
%This assumption is implied in Sarig, (Theorem 2, and the last paragraph
%in page 647,)
%and stated explicitly in Gouezel, (Theorem 1.3.)}
\end{enumerate}
\end{AssumptionT}

\begin{Remark}\label{RmkA4}
Conditions (b) and (c) in Assumption T correspond to
conditions (F) and (U) in \cite{ADSZ}. There is there a third assumption, (A), which is distortion and which is not necessarily guaranteed in our systems.
With this precision, we could  regard the systems satisfying Assumption~T(a)-(c)
as higher dimensional ``AFU'' systems. Returning to the finite image condition T (b), it is used in proof of Lemma~2.1 below,
to get $\mu(A_{n, n_0})\to 1$ as $n_0\to \infty$ and this is a consequence of a ``small image boundary'' as explained in the first footnote of the proof of Theorem B.

\end{Remark}

\begin{Remark}\label{RmkA5}
We mention that if $T$ has relatively prime return times on
almost all points $x\in\wh X$, then Condition (d) is satisfied.
The former means that for any neighborhood $U$ of $x$, there is a point
$y\in U$ and return times $\tau'(x)$ and $\tau'(y)$
such that $T^{\tau'(x)}(x),T^{\tau'(x)}(x)\in U$
and the greatest common divisor $(\tau'(x), \tau'(y))=1$.
Here $\tau'(x)$ and $\tau'(y)$ are not necessary the first return time.
\end{Remark}

%Also we put some more assumptions on the Banach space $\B$.

%It follows that if $U$ is almost open, then there exists
%a open set $U'$ such that $\nu(U\D U')=0$, where $U\D U'$ denotes
%the symmetric difference of $U$ and $U'$.  In fact, we can find
%countable dense set $\{x_i\}$ in $U$ such that for every $x_i$,
%there is neighborhood $V(x_i)$ with $\nu(V(x_i)\setminus U)=0$.
%It is easy to see that $U':=\cup_{i=1}^\infty V(x_i)$ is a such set.

%\begin{AssumptionB}\label{AB1}
%\begin{enumerate}

%\end{enumerate}
%\end{AssumptionB}

Let us take now a partition $\xi$ of $\wh X$ and
consider a family of skew-products of the form
\begin{equation}\label{fdefskew}
\wt T=\wt{T}_S: \wh{X}\times Y \rightarrow \wh{X}\times Y \ ,
\ \wt{T}_S(x,y)=\bigl(\wh{T}x, \; S(\xi(x))(y)\,\bigr),
\end{equation}
where $(Y, {\cal F}, \rho)$ is a Lebesgue probability space,
$\Aut(Y)$ is the collection of its automorphisms, that is,
invertible measure-preserving transformations, and
$S:\xi\rightarrow \Aut(Y)$ is arbitrary.
%%%%%%%%%%%%%%%%%%%%%%%%%%%%%%%%%%%%%%%%%%%%%%%%%%
%The fibred system is called {\em skew-product rigid} if for every
%invariant density $h(x,y)$ of an arbitrary skew-product over $\xi$,
%$\{h(\cdot, y)>0\}$ is almost open for a.e. $y\in Y$.
%%%%%%%%%%%%%%%%%%%%%%%%%%%%%%%%%%%%%%%%%%%%%%%%%%

We then consider functions $\wt f\in L^1({\hn \times \rho})$ and define
\begin{equation*}%\label{fdeftnorm}
|\wt f|_{\wt{\B}}=\int_Y|\wt f(\cdot, y)|_{\B}d\rho(y), \qquad
\|\wt f\|_{\wt{\B}}=|\wt f|_{\wt{\B}}+\|\wt f\|_{L^1({\hn \times \rho})}.
\end{equation*}
Then we let
\begin{equation*}%\label{fdeftB}
\wt \B=\{\wt f\in L^1({\hn \times \rho}):  |\wt f|_{\wt{\B}}< \infty \}.
\end{equation*}
It is easy to see that with the norm $\|\cdot \|_{\wt{\B}}$,
$\wt \B$ is a Banach space.

The transfer operator $\wt{\P}=\wt{\P}_{\hn\times \rho}$ acting
on $L^1(\hn \times \rho)$ is
defined as the dual of the operator $\wt f \to \wt f\circ \wt T$
from $L^\infty(\hn \times \rho)$ to itself.
Note that if $Y$ is a space consisting of a single point, then
we can identify $\wh X\times Y$, $\wt T$ and $\wt \P$ with
$\wh X$, $\wh T$ and $\wh \P$ respectively.

\begin{TheoremB}\label{ThmB}
Let us suppose $\wh T$ satisfies Assumption T(a) to (d) and $\B$ satisfies
Assumption B(d) to (f), and $\wt \P$ satisfies the  Lasota-Yorke inequality
\begin{equation}\label{LYineq3}
|(\wt{\P}\wt f)|_{\wt \B}
\le \wt\eta|\wt f|_{\wt \B}+\wt D\|\wt f\|_{L^1({\hn \times \rho})}
\end{equation}
for some $\wt\eta\in (0,1)$ and $\wt D>0$.
Then the absolutely continuous invariant measure $\hmu$ obtained
from the Lasota-Yorke inequality \eqref{LYineq2} is ergodic
and $e^{it\tau}$ is aperiodic.
Therefore Assumptions $(S2)$ and $(S3)$ follow.
\end{TheoremB}

\begin{Remark}\label{RmkB3}
It is well known that for $C^{1+\a}$, $\a>1$,  uniformly expanding maps or
uniformly hyperbolic diffeomorphisms, the absolutely continuous
invariant measures or the SRB measure $\mu$ are ergodic
if the maps are topological mixing, see e.g. \cite{Bo} for invertible case; the noninvertible case can be obtained similarly.

However, it is not the case if the conditions on $C^{1+\a}$
or uniformity of hyperbolicity fail.
In \cite{Qu} the author gives an example of $C^1$ uniformly expanding maps
of the unit circle, and in \cite{HPT} the authors provide an example of
$C^\infty$ diffeomorphisms, where the Lebesgue measure is preserved
and topological mixing does not give ergodicity.
In the proof of the theorem we in fact give some additional
conditions under which topological mixing implies ergodicity
(see Lemma~\ref{LB3}).
\end{Remark}

%%%%%%%%%%%%%%%%%%%%%%%%%%%%%%%%%%%%%%%%%%%%%%%%%%%%%%%%%%%%%%%
%%%%%%%%%%%%%%%%%%%%%%%%%%%%%%%%%%%%%%%%%%%%%%%%%%%%%%%%%%%%%%%

%%%%%%%%%%%%%%%%%%%%%%%%%%%%%%%%%%%%%%%%%%%%%%%%%%%%%%%%%%%%%%%
%%%%%%%%%%%%%%%%%%%%%%%%%%%%%%%%%%%%%%%%%%%%%%%%%%%%%%%%%%%%%%%
\section{Aperiodicity}\label{SThmB}
\setcounter{equation}{0}
%%%%%%%%%%%%%%%%%%%%%%%%%%%%%%%%%%%%%%%%%%%%%%%%%%%%%%%%%%%%%%%
%%%%%%%%%%%%%%%%%%%%%%%%%%%%%%%%%%%%%%%%%%%%%%%%%%%%%%%%%%%%%%%

The proof of Theorem B is based on a result in \cite{ADSZ}. We
briefly mention the terminology used there.

A {\it fibred system} is a quintuple $(X, \A, \nu, T, \xi)$,
where $(X, \A, \nu, T)$ is a nonsingular transformation on a
$\sigma$-finite measure space and $\xi\subset \A$ is a finite or
countable partition (mod $\nu$) such that:

\begin{enumerate}
\item[{\rm (1)}]
$\xi_\infty=\bigvee_{i=0}^\infty T^{-i}\xi$ generates $\A$;

\item[{\rm (2)}]
every $A\in \xi$ has positive measure;

\item[{\rm (3)}]
for every $A\in \xi$, $T|_A: A\to TA$ is bimeasurable invertible
with nonsingular inverse.
\end{enumerate}

The transformation given in \eqref{fdefskew} is called the
{\it skew product over $\xi$}. We will denote with $\xi_n$ the $n$-join $\xi_n=\bigvee_{i=0}^{n-1} T^{-i}\xi,$ and with $\xi_n(x)$ the element (cylinder) of the partition $\xi_n$ containing the point $x$. Consider the corresponding
transfer operator  $\wt \P=\wt \P_{\nu\times\rho}$.
A fibred system $(X, \A, \nu, T, \xi)$ with $\nu$ finite is called
{\it skew-product rigid} if for every invariant function $\wt h(x,y)$
of $\wt \P$ of an arbitrary skew product $\wt T_S$, the set
$\{\wt h(\cdot, y)>0\}$ is almost open $(\mod \nu)$
for almost every $y\in Y$.
In \cite{ADSZ}, a set $U$ being almost open $(\mod \nu)$ means that
for $\nu$ almost every $x\in U$,
there is a positive integer $n$ such that $\nu(\xi_n(x)\setminus U)=0$.
%It is easy to see that in a fibred system, this is equivalent to
%the definition we gave before Assumption B.
Since the partition $\xi$ we are interested in satisfies
$\nu(\partial A)=0$ for any $A\in \xi_n$ and $\wh T$ is piecewise smooth,
the fact that $\xi_\infty$ generates $\mathcal A$ implies
that the definition given there is the same as we defined
for Assumption B(f).

A set that can be expressed in the form $\wh T^n \xi_n(x)$, $n\ge 1$ and
$x\in {\wh X}$, is called an {\it image set}.
A cylinder $C$ of length $n_0$ is called a {\it cylinder of full returns},
if for almost all $x\in C$ there exist $n_k\nearrow \infty$ such that
$\wh T^{n_k}\xi_{n_k+n_0}(x)= C$.  In this case we
say that $\wh T^{n_0}(C)$ is a {\it recurrent image set}.
%A fibred system is called {\it almost onto} if any measurable map
%$f: X\to \mathbb R$ which is constant almost everywhere restricted
%to every recurrent image set is constant almost everywhere on $X$.

Our proof of Theorem B is based on a result given in Theorem 2
in \cite{ADSZ}:

\begin{Theorems}\label{ThmADSZ}
Let $(X, \A, \mu, T, \xi)$ be a skew-product rigid measure preserving
fibred system whose image sets are almost open.
Let $\mathbb{G}$  be a locally compact Abelian polish group.
If $\c\circ \phi =\l f/f\circ T$ holds almost everywhere,
where $\phi: X\to {\mathbb G}$, $\xi$ measurable,
$\c\in \wh {\mathbb G}$, $\l\in {\mathbb S}$,
then $f$ is constant on every recurrent image set.
\end{Theorems}

{\em Warning}: In the proof of Theorem B and the lemmas below we will work exclusively
on the induced space $\wh X$ and with measures $\hn$ and $\hmu$ and density
$\h;$ for this reason  we will drop the hat on those notations.

\begin{proof}[Proof of Theorem B]

Recall that $\mu$ is an $\wh T$ invariant measure with density $h$,
where $h$ is the fixed point of $\wh \P$ in $\B$.
By Lemma~\ref{LB3} we know that $\mu$ is ergodic;
hence we only need to prove that $e^{it\tau}$ is aperiodic.

Let us denote with ${\mathcal A}$ the Borel $\s$-algebra inherited from $\mathbb R^m$ and
take a countable partition $\xi$ of ${\wh X}$ into $\{U_i\}$ or finer.
We also require that each $A\in \xi$ is almost open,
and $\nu( B_\e(\partial \wh T\xi))=O(\e)$, where
$\partial \wh T\xi=\cup_{A\in \xi}\partial (\wh TA)$.
 Is it obvious that  we can take smooth surfaces
as the boundary of the elements of $\xi$, in addition to Assumption T(b) \footnote{This assumption is in fact used to get the measure of an $\epsilon$-neighborhood of the boundary of $\hat{T}\xi$ of order $\epsilon$.}.
Since $\wh T$ is uniformly expanding by Assumption T(c),
we know that each element of $\xi_\infty=\bigvee_{i=0}^\infty \wh T^{-i}\xi$
contains at most one point.\footnote{In fact, if $x,y\in \xi_{\infty}$, then for any $i>0$, $\wh T^i x$
and $\wh T^i y$ are always in the same elements of $\xi$,
and hence in the same $U_{n_i}$ for some $n_i>0$.  On the other hand,
by Assumption~T(c) we have $d(\wh T^i x, \wh T^i y)\ge s^{-i}d(x,y)$.
If $d(x,y)\not= 0$, then $d(\wh T^i x, \wh T^i y)\to \infty$, contradicting
the facts that $\wh X$ is compact. We in fact recall that in  Lebesgue spaces a necessary and sufficient condition for $\xi_n\rightarrow \mathcal{A}$ is that there exists a set of zero measure $N\subset \wh X$ such that for $x,y\in \wh X/N$ (with $x\neq y$) there exists $n\ge 1$ and $U\in \xi_n$ such that $x\in U$ but $y\notin U$.}  Therefore $\xi_\infty$ generates $\mathcal A$.
We may regard that each $A\in \xi$ has positive measure, otherwise
we can use ${\wh X}\setminus A$ to replace ${\wh X}$.
Also, for every $A\in \xi$, $\wh T|_A: A\to \wh TA$ is a diffeomorphism,
and therefore $\wh T|_A$ is bimeasurable invertible with nonsingular
inverse.
Hence the quintuple $({\wh X}, \mathcal A, \mu, \wh T, \xi)$ is a
measure preserving fibred system.

The construction of $\xi$ implies that
$\mu(\partial \xi)=\nu(\partial \xi)=0;$
therefore $\mu(\partial \xi_n)=\nu(\partial \xi_n)=0$ for any $n\ge 1$.
We point out  that the intersection of finite number of almost open sets
is still almost open.
Moreover, the differentiability of $\wh T$ on each $U_i$ implies that
all elements $\xi_n(x)$ of $\xi_n$ are almost open, and therefore
all image sets $\wh T^n\xi_n(x)$ are almost open with respect to $\mu$.

To get skew product rigidity, let us consider the skew product $\wt T_S$
defined in (\ref{fdefskew}) for any $(Y, {\cal F}, \rho)$.
Let $\wt \P=\wt \P_{\nu\times \rho}$ be the transfer operator
and $\wt h$ an invariant function, that is, $\wt \P\wt h=\wt h$.
By Proposition~\ref{Pspr} below we know that $\wt h\in \wt B$.
Hence, for $\rho$-almost every $y\in Y$, $\wt h(\cdot, y)\in \B$.
By Assumption~B(f), $\{\wt h(\cdot, y)>0\}$ is almost open mod $\nu$.
This gives the skew product rigidity.

So far we have verified all conditions in the theorem of \cite{ADSZ}
stated above.
Applying the theorem to the equation $e^{it\tau}=f/f\circ \wh T$
almost everywhere, where $f: {\wh X}\to {\mathbb S}$ is a measurable function,
we get that $f$ is constant on every recurrent image sets $J$.

Now we prove aperiodicity, by following similar arguments in \cite{Go}.
Let us assume that the equation $e^{it\tau}=f/f\circ \wh T$ holds almost everywhere
for some real number $t$ and a measurable function $f: {\wh X}\to {\mathbb S}$.
By Lemma~\ref{LB2} below we get that ${\wh X}$ contains a recurrent
image set $J$ with $\mu (J)>0$ and by the theorem above, we know
that $f$ is constant, say $c$, almost everywhere on $J.$ Then by the absolute continuity of $\mu$ and the fact that $\{h>0\}$ is $\nu$-almost open, we can find an open set $J'\subset J$ of positive $\mu$-measure.
Thanks to  Assumption T(d), $T$ is topological mixing and
therefore for all sufficiently large $n$, we have $T^{-n}J'\cap J'\not=\emptyset$.
Since the intersection is open\footnote{Strictly speaking that intersection contains open sets since $T$ and all its powers, although  not continuous, are local diffeomorphisms, on each domain where they are injective.} , we get that $\mu(T^{-n}J'\cap J')>0$ and as a consequence
 for any typical point $x$ in $T^{-n}J'\cap J'$, there is $k>0$
such that $T^nx=\wh T^kx$, and $n=\sum_{i=0}^{k-1}\tau(\wh T^i x)$.
Since $e^{it\tau}=f/f\circ \wh T$ along the orbit of $x$, we have
$$
e^{int}=e^{it\sum_0^{k-1}\tau(\wh T^i x)}
=\frac{f(x)}{f(\wh Tx)}\frac{f(\wh Tx)}{f(\wh T^2x)}
\cdots \frac{f(\wh T^{k-1}x)}{f(\wh T^kx)}
=\frac{f(x)}{f(\wh T^kx)}
=\frac{c}{c}
=1.
$$
Since this is true for all large $n$, by replacing $n$ by $n+1$
we get that $e^{it}=1$.  It follows that $t=0$ and  $f=f\circ \wh T$
almost everywhere which implies that   $f$ must be a constant almost everywhere
since $\mu$ is ergodic.
\end{proof}

To prove Lemma~\ref{LB2}, we need a result from Lemma 2 in Section 4
in \cite{ADSZ}.  We state it as the next lemma.
The setting for the lemma is a conservative fibred system and
it can be applied directly to our case.

\begin{Lemmas}\label{LemADSZ}
A cylinder $C\in \xi_{n_0}$ is a cylinder of full returns
if and only if there exists a set $K$ of positive measure
such that for almost every $x\in K$, there are $n_i\to \infty$
with $\wh T^{n_i}\xi_{n_i+n_0}(x) = C$.
\end{Lemmas}

\begin{Lemma}\label{LB2}
There is a recurrent image set $J$ contained in ${\wh X}$ with $\mu (J)>0$.
\end{Lemma}

\begin{proof}
We first recall that $s$ is given in Assumption T(c); then let us
take $C_\xi>0$ such that $\diam D\le C_\xi$ for all $D\in \xi$ and set
\begin{eqnarray*}
&&A_{k,n_0}'
=\{x\in {\wh X}: x\notin B_{C_\xi s^{k+n_0}}(\partial \wh T \xi)\},  \\
&&A_{n,n_0}=\bigcap_{k=0}^{n-1} \wh T^{n-k} A_{k,n_0}'.
\end{eqnarray*}
By the construction of $\xi$, there is $C'>0$ such that
$\nu (A_{k,n_0}')\ge 1-C'C_\xi s^{k+n_0}$; moreover
assumption~B(b) guarantees that  $\|h\|_\infty <\infty$.
Therefore if we take $C=C'C_\xi \|h\|_\infty/(1-s)$, then
$\mu (A_{k,n_0}')\ge 1-C'C_\xi \|h\|_\infty s^{k+n_0}= 1-C(1-s) s^{k+n_0}$.
Since $\mu$ is an invariant measure,
$\mu( A_{n,n_0})\ge 1-C (1-s)\sum_{i=0}^{n-1}s^{i+n_0}
\ge 1-Cs^{n_0}$.
If we choose $n_0$ large enough, then $\mu (A_{n,n_0})$ is
bounded below by a positive number for all $n>0$,
and the bound can be chosen arbitrarily close to $1$ by taking $n_0$
sufficiently large.

Note that $\xi_n$ is a partition with at most countably many elements.
For each $n_0>0$, let $B'_{n_0}$ be the union of a finite number of  elements
of $\xi_{n_0}$ such that $\mu (B'_{n_0})> 1-C s^{n_0}/2$.
Then set $B_{n, n_0}=B'_{n_0}\cap \wh T^{-n}B'_{n_0}$;
clearly, $\mu( B_{n,n_0})\ge 1-Cs^{n_0}$.
If we now put $C_{n,n_0}=A_{n, n_0}\cap B_{n, n_0},$ then
we have $\mu (C_{n,n_0})\ge 1-2Cs^{n_0}$.
Hence, $\sum_{n=0}^\infty \mu (C_{n,n_0})=\infty$ for all
large $n_0$.

A generalized Borel-Cantelli Lemma by Kochen and Stone
(see \cite{Ya}),
gives that for any given $n_0>0$, the set of points that belong to
infinitely many $C_{n,n_0}$ has the measure bounded below by
\begin{eqnarray*}%\label{FBCL}
\limsup_{n\to\infty} \frac{\sum_{1\le i<k\le n}\mu (C_{i,n_0})\mu (C_{k,n_0})}
{\sum_{1\le i<k\le n}\mu (C_{i,n_0}\cap C_{k,n_0})}.
\end{eqnarray*}
Notice that if $n_0\to \infty$, then  $\mu (C_{i,n_0}) \rightarrow 1$ as $n_0\rightarrow \infty$ and uniformly in $i$ by the previous lower bound on $\mu(C_{n,n_0}).$  Hence the upper limit goes to $1$
as $n_0\to \infty$.
If we now se
$$\C_{n_0}=\{x\in {\wh X}: x\in C_{n,n_0} \ \mbox{infinitely often}\},
$$
the above argument gives $\mu (\C_{n_0})\to 1$ as $n_0\to \infty$.

We observe that for a one-to-one map $T$, $T(A\cap T^{-1}B)=B$
if and only if $B\subset TA$.
Since $\xi_n(x)=\xi(x)\cap \wh T^{-1}(\xi_{n-1}(\wh Tx))$,
and $\wh T$ is a local diffeomorphism, we know that
$\wh T\xi_n(x)=\xi_{n-1}(\wh Tx)$ if and only if
$\xi_{n-1}(\wh Tx)\subset \wh T \xi(x)$.
Inductively, $\wh T^n\xi_{n+n_0}(x)=\xi_{n_0}(\wh T^n x)$
if and only if $\xi_{n-i+n_0}(\wh T^i x)\subset \wh T\xi(\wh T^{i-1}x)$
for $i=1, \cdots, n$.  If $x\in A_{n,n_0}$ for some $n,n_0>0$,
then $\wh T^{n-i}x\notin B_{C_\xi s^{i+n_0}}(\partial \wh T\xi)$
for all $i=1, \cdots, n$.
Since the diameter of each member of $\xi$ is less than $C_\xi$,
by Assumption T(c), $\diam \xi_n(x)\le C_\xi s^n$ for any $x\in {\wh X}$
and $n\ge 0$.
We get $\xi_{n-i+n_0}(\wh T^i x)\subset \wh T\xi(\wh T^{i-1}x)$
and therefore $\wh T^n\xi_{n+n_0}(x)=\xi_{n_0}(\wh T^nx)$.
Consequently, if $x\in \C_{n_0}$, then
$x\in C_{n_i,n_0}=A_{n_i,n_0}\cap B_{n_i,n_0}$ for infinitely many
$n_i$.  Hence, $\wh T^{n_i}\xi_{n_i+n_0}(x)=\xi_{n_0}(\wh T^{n_i}x)$
and $\wh T^{n_i}x\in B_{n_0}$ for infinitely many $n_i$,

We now take $n_0>0$ such that $\mu (\C_{n_0})>0;$
since $B_{n_0}$ consists of only finitely many elements in $\xi_{n_0}$,
we know that there is an element $C\in \xi_{n_0}$ with $C\subset B_{n_0}$
such that
\begin{equation}\label{fmuC}
\mu\{x: \wh T^{n}\xi_{n+n_0}(x)=\xi_{n_0}(\wh T^{n}x)=C \
\mbox{infinitely often}\}>0.
\end{equation}
By the above lemma from \cite{ADSZ}, $C$ is a cylinder of full returns.
Hence, $J=\wh T^{n_0}C$ is a recurrent image set.
Since $\mu$ is an invariant measure, \eqref{fmuC} implies $\mu(C)>0$
and therefore $\mu(J)>0$.
\end{proof}

\begin{Lemma}\label{LB3}
Let us suppose that $T$ and $\B$ satisfy Assumption T(d) and B(f) respectively.
Then there is only one  absolutely continuous invariant measure $\mu$ which is  ergodic.
\end{Lemma}

\begin{proof}
Suppose $\mu$ has two ergodic components $\mu_1$ and $\mu_2$ whose
density functions are $h_1$ and $h_2$ respectively.
Hence, $\nu(\{h_1>0\}\cap \{h_2>0\})=0$.
Since $h_1,h_2\in \B$, the sets $\{h_1>0\}$ and $\{h_2>0\}$ are almost open.
We can take open sets $U_1$ and $U_2$ such that
$\nu(U_1\setminus \{h_1>0\})=0$ and $\nu(U_2\setminus \{h_1>0\})=0$.
Since $T$ is topological mixing, there is $n>0$ such that
$T^{-n}U_1\cap U_2\not=\emptyset$.  Hence, $\nu(T^{-n}U_1\cap U_2)>0$
and therefore $\nu(U_1\cap T^{n}U_2)>0$.
It follows that there is $k>0$ such that $\nu(U_1\cap \wh T^{k}U_2)>0$.
Since $\hat{\P} h_2=h_2$, $h_2(x)>0$ implies $h_2(\wh T^{k}x)>0$.
Hence $\nu(\wh T^{k}U_2\setminus \{h_2>0\})=0$.
Therefore, $\nu(\{h_1>0\}\cap \{h_2>0\})\ge \nu(U_1\cap \wh T^{k}U_2)>0$, which is
a contradiction.
\end{proof}

We are left with the proof that any fixed point $\tilde{h}$ of $\tilde{\cal P}$ belongs to ${\cal B}.$
The result was proved for Gibbs-Markov maps in \cite{AD}.
We show that it holds in more general cases. To simplify the notation we will write  from now on $L^1(\nu\times \rho)$ instead of $L^1(\wh X\times Y, \nu\times \rho)$.

\begin{Proposition}\label{Pspr}
Suppose that $\B$ satisfies Assumption B(d) and (e), and
$\wt \P$ satisfies Lasota-Yorke inequality \eqref{LYineq3}.
Then any $L^1(\nu \times \rho)$ function $\wt h$ on $\wh X\times Y$
that satisfies $\wt \P_{\nu\times \rho}\wt h=\wt h$ belongs to $\wt\B$.
\end{Proposition}

\begin{proof}
By Assumption B(d), $\B$ is dense in $L^1(\wh X, \nu)$;
it is easy to see that $\wt\B$ is dense in
$L^1(\nu\times \rho)$ too.
Hence, for any $\e>0$ we can find a nonnegative function
$\wt f_{\e}\in \wt \B$ such that
$\|\wt f_{\e}-\wt h||_{L^1({\nu\times \rho})}<\e$.
By the stochastic ergodic theorem, see  Krengel (\cite{Kr}), there exists
a nonnegative function $\wt h_\e\in L^1(\nu\times \rho)$
and a subsequence $\{n_k\}$ such that
\begin{equation}\label{flimhe}
\lim_{k\to \infty} \frac{1}{n_k}\sum_{\ell=0}^{n_k-1}
\wt{\P}^\ell \wt f_{\e}=\wt h_{\e}  \qquad \nu\times\rho\mbox{-a.e.}
\end{equation}
and $\wt{\P}\wt h_{\e}=\wt h_{\e}$.

Notice that Lasota-Yorke inequality \eqref{LYineq3} implies that
for any $\wt f\in \wt\B$, $\ell \ge 1$,
\begin{equation}\label{LYineqPl}
|\wt{\P}^\ell\wt f|_{\wt \B}
\le \wt\eta^\ell|\wt f|_{\wt \B}+\wt{D}^*\|\wt f\|_{L^1({\nu \times \rho})}
\le \wt D_2 \|\wt f\|_{\wt \B},
\end{equation}
where $\wt{D}^*=\wt D \wt\eta/(1-\wt\eta)
\ge \wt D(\wt\eta+\dots +\wt\eta^{\ell -1})$
and $\wt D_2=1+\wt{D}^*$.
Denote $\disp \psi_k= \frac{1}{n_k}\sum_{\ell=0}^{n_k-1}\wt{\P}^\ell f_{\e};$ by \eqref{LYineqPl} $\psi_k\le \wt D_2 \|\wt f\|_{\wt \B}$. On the other hand
\eqref{flimhe} implies that
$\disp \liminf_{k\to \infty} \psi_k(x,y)=\wt h_\e(x,y)$ for $\nu$-a.e.
$x\in \wh X$, $\rho$-a.e. $y\in Y$.
Hence, by Assumption B(e) and Fatou's lemma we obtain
\begin{equation}\label{fFatou1}
\begin{split}
|\wt h_{\e}|_{\wt{\B}}
=&\int_Y |\lim_{k\to\infty} \psi_k(\cdot, y)|_{\B}d\rho(y)
\le\int_Y \liminf_{k\to\infty} |\psi_k(\cdot, y)|_{\B}d\rho(y) \\
\le&\liminf_{k\to\infty} \int_Y |\psi_k(\cdot, y)|_{\B}d\rho(y)
= \liminf_{k\to \infty} |\psi_k|_{\wt{\B}}\le \wt D_2||\wt f_{\e}||_{\wt{\B}}.
\end{split}
\end{equation}
This means that $\wt h_{\e}\in {\wt{\B}}$.

By Fatou's Lemma and the fact that $\wt{\P}$ is a contraction on
$L^1(\nu\times \rho)$, it follows immediately that
\eqref{flimhe} and  $\wt\P \wt h=\wt h$ imply
$$
\|\wt h-\wt h_{\e}\|_{L^1({\nu\times \rho})}
\le \liminf_{k\to \infty}\frac{1}{n_k}\sum_{l=0}^{n_k-1} ||
            \wt{\P}^\ell(\wt h-\wt f_{\e})\|_{L^1({\nu\times \rho})}
\le \|\wt h-\wt f_{\e}\|_{L^1({\nu\times \rho})}
\le\e.
$$
By the first inequality of \eqref{LYineqPl} we know that for any $n\ge 1$,
$$
\|\wt h_{\e}\|_{\wt{\B}}=\|\wt{\P}^n \wt h_{\e}\|_{\wt{\B}}
\le \wt\eta^n\|\wt h_{\e}\|_{\wt{\B}}+\wt{D}^*\|\wt h_{\e}\|_{L^1({\nu\times \rho})}.
$$
If we now send  $n$ to infinity we get
$\|\wt h_{\e}\|_{\wt{\B}}\le \wt{D}^*\|\wt h_{\e}\|_{L^1({\nu\times \rho})}
\le \wt{D}^*(\|\wt h\|_{L^1({\nu\times \rho)}}+\e)$.
We then replace $\e$ with a decreasing sequence $c_n\rightarrow 0$
as $n\to \infty$.  Since $\wt h_{c_n}$ converges in $L^1({\nu\times \rho})$
to $\wt h$, there is a subsequence $n_i$ such that
$\lim_{i\to \infty} \wt h_{c_{n_i}}= \wt h$, $\nu\times\rho$-a.e..
Then by the same arguments used in  \eqref{fFatou1}, we get
$$
|\wt h-\wt h_{c_n}|_{\wt{\B}}
\le \liminf_{i\to \infty}|\wt h_{c_{n_i}}-\wt h_{c_n}|_{\wt{\B}}
\le 2 \sup_{0\le \e\le 1}\|\wt h_{\e}\|_{\wt{\B}}
\le 2\wt D_1(\|\wt h\|_{L^1({\nu\times \rho})}+1).
$$
We have thus  obtained $\wt h-\wt h_{c_n}\in \wt\B$ and as a consequence  $h=(h-h_{c_n})+h_{c_n}\in {\wt{\B}}$ and this completes the proof.
\end{proof}

%%%%%%%%%%%%%%%%%%%%%%%%%%%%%%%%%%%%%%%%%%%%%%%%%%%%%%%%%%%%%%%
%%%%%%%%%%%%%%%%%%%%%%%%%%%%%%%%%%%%%%%%%%%%%%%%%%%%%%%%%%%%%%%
%\bigskip
%\bigskip
%\noindent{\bf \Large Part II: Applications to Non-Markov Maps}

%%%%%%%%%%%%%%%%%%%%%%%%%%%%%%%%%%%%%%%%%%%%%%%%%%%%%%%%%%%%%%%
%%%%%%%%%%%%%%%%%%%%%%%%%%%%%%%%%%%%%%%%%%%%%%%%%%%%%%%%%%%%%%%

%%%%%%%%%%%%%%%%%%%%%%%%%%%%%%%%%%%%%%%%%%%%%%%%%%%%%%%%%%%%%%%
%%%%%%%%%%%%%%%%%%%%%%%%%%%%%%%%%%%%%%%%%%%%%%%%%%%%%%%%%%%%%%%
\section{Systems on the interval}\label{SSonedim}
\setcounter{equation}{0}
%%%%%%%%%%%%%%%%%%%%%%%%%%%%%%%%%%%%%%%%%%%%%%%%%%%%%%%%%%%%%%%
%%%%%%%%%%%%%%%%%%%%%%%%%%%%%%%%%%%%%%%%%%%%%%%%%%%%%%%%%%%%%%%

%\subsection{Setting and Statement of results.}%\label{SSmapft}

In this section we take $X=[0,1]$ and $\nu$ be the Lebesgue measure on $X$.

We remind  that for a map $T: X\to X$ and a subset $\wh X\subset X$,
the corresponding first return map is denoted by $\wh T: \wh X\to \wh X$; $\hn$ will be again the normalized Lebesgue measure on $\wh X$.

Let us now assume that $T: X\to X$ is a  map satisfying the following conditions.

\begin{AssumptionT'}
\begin{enumerate}
\item[{\rm (a)}] {\rm (Piecewise smoothness)}
There are points $0=a_0<a_1<\cdots <a_K=1$ such that for each $j$,
$T_j=T|_{I_j}$ is a $C^{2}$ diffeomorphism on its image,
where $I_j=(a_{j-1}, a_j)$.

\item[{\rm (b)}] {\rm (Fixed point)}
$T(0)=0$.

\item[{\rm (c)}] {\rm (Expansion)}
There exists $z\in I_1$ such that $T(z)\in I_1$ and
$\disp \D:=\inf_{x\in \wh X}|T'(x)|>2$ for any
$x\in \wh X$, where $\wh X=[z,1]$.

\item[{\rm (d)}] {\rm (Distortion)}
%$\disp D:=\sup_{x\in [0,z]}\frac{|\wh T''(x)|}{|\wt T'(x)|^2}\le \infty$,
$\disp \C:=\sup_{x\in [z,1]}|\wh T''(x)|/|\wh T'(x)|^2< \infty$.
%where $\wh T$ is the first return map of $T$ with respect to $\wh X$.

\item[{\rm (e)}] {\rm (Topological mixing)}
$T: I\to I$ is topological mixing.
\end{enumerate}
\end{AssumptionT'}

%Since we are interested in absolutely continuous invariant measures,
%we usually ignore the end point of subintevals.

We now set $J=[0,z)$ and ${\wh X}={\wh X}_J=X\setminus J$.
$I_0=TJ\setminus J\subset I_1$.
We also denote the first return map $\wh T=\wh T_J$
by $\wh T_{ij}$ if $\wh T=T_1^iT_{j}$.
Further, we put $I_{01}=I_1\setminus J$,
$I_{0j}=I_j\setminus T_j^{-1}J$ if $j>1$,
and $I_{ij}=\wh T^{-1}_{i,j}I_0$ for $i>0$.
Hence, $\{I_{ij}: i=0,1,2,\cdots\}$ form a partition of $I_j=(a_j,b_j)$
for $j=2,\cdots, K$.
Also, we denote $\bar I_{ij}=[a_{ij}, b_{ij}]$
for any $i=0,1,2, \cdots$ and $j=1, \cdots, K$.

Recall that the  variation of a real or complex valued
function $f$ on $[a,b]$ is defined by
$$
V_{[a,b]}(f)
:=\sup_{\xi\in \Xi}\sum_{i=1}^{n}|f(x^{(\ell)})-f(x^{(\ell -1)})|,
$$
where $\xi$ is a finite partition of $[a,b]$ given by
$a=x^{(0)}<x^{(1)}<\cdots <x^{(n)}=b$ and $\Xi$ is the set of all such partitions. A function $f\in L^1([a,b],\nu)$, where $\nu$ denotes the Lebesgue measure, is of bounded variation if $ V_{[a,b]}(f)= \inf_{g}V_{[a,b]}(g)<\infty$, where the infimum is taken over all the functions $g=f$ $\nu$-a.e..
Let $\B$ be the set of functions $f \in L^1(\wh X, \hn), \ f: \wh X\to \mathbb R$ with
$V_{\wh X}(f)<\infty$.
For $f\in \B$, denote by $|f|_\B=V_{\wh X}(f)$, the total variation of $f$.
Then we define $\|f\|_\B=\|f\|_1+|f|_\B$, where the $L^1$ norm is intended with respect to $\hn$.
It is well known that $\|\cdot \|_\B$ is a norm, and with this norm,
$\B$ becomes a Banach space.

To obtain the decay rates, we also assume that
there are constants $0<\c<1$, $\c'>\c$ and $\tilde{C}>0$
such that in a neighborhood of the indifferent fixed point $p=0$,
\begin{equation}\label{A.0}
\begin{split}
 &T(x)=x+\tilde{C}x^{1+\c}+O(x^{1+\c'}),  \\
 &T'(x)=1+\tilde{C}(1+\c)x^\c+O(x^{\c'}),  \\
 &T''(x)=\tilde{C}\c(1+\c)x^{\c-1}+O(x^{\c'-1}).
\end{split}
\end{equation}
%\begin{eqnarray}%\label{B.5}
% &T(x)=x\bigl(1+Cx^\c+O(x^{\c'})\bigr),  \label{A.0} \\
% &T'(x)=1+C(1+\c)x^\c+O(x^{\c'}),  \label{A.1} \\
% &T''(x)=C\c(1+\c)x^{\c-1}+O(x^{\c'-1}).  \label{A.2}
%\end{eqnarray}

For any sequences of numbers $\{a_n\}$ and $\{b_n\}$, we write
$a_n\sim b_n$ if $\disp\lim_{n\to \infty}a_n/b_n=1$, and
$a_n\approx b_n$
if $c_1b_n\le a_n\le c_2b_n$ for some constants $c_2\ge c_1>0$.

We now set:
\begin{equation}\label{fdefdij}
d_{ij}=\sup\{|\wh T_{ij}'(x)|^{-1}: x\in I_{ij}\},  \quad
d_n=\max\{d_{n, j}:  2\le j\le K\}.
\end{equation}

\begin{TheoremC}\label{ThmC}
Let $\wh X$, $\wh T$ and $\B$ be defined as above, and
suppose that $T$ satisfies Assumption T$\,{}'$~(a) to (e).
Then Assumption B(a) to (f)
and assumptions $(S1)$ to $(S3)$  are satisfied and
$\|R_n\|= O(d_n)$.
Hence, if $d_n=O(n^{-(\b+1)})$ for some $\b>1$,
then there exists $C>0$ such that for any functions
$f\in \B$, $g\in L^\infty(X,\nu)$ with
$\supp{f}, \; \supp{g}\subset {\wh X}$,
(\ref{fThmA}) holds.

In particular, if $T$ satisfies \eqref{A.0} near $0$,
then $\disp\sum_{k=n+1}^\infty \mu(\tau>k)=O(n^{-(\frac{1}{\c} -1)})$
and $d_n=O(n^{-(\frac{1}{\c}+1)})$. Since $\frac{1}{\c}-1<\frac{1}{\c}$ and $\frac{1}{\c}-1<2(\frac{1}{\c}-1)$
 we have
\begin{equation*}%\label{fThmD2}
\Cov(f, g\circ T^n)
\sim \sum_{k=n+1}^\infty \mu(\tau>k)\int fd\mu\int gd\mu\
\approx \frac{1}{n^{\frac{1}{\c} -1}}.
\end{equation*}
\end{TheoremC}

It is well known that if the map $T$ allows a Markov partition, then
the  decay of correlations is of order $O(n^{-(\frac{1}{\c}-1)})$
(see e.g. \cite{Hu}, \cite{Sr}, \cite{LSV}, \cite{PY}).
For non-Markov case, the upper bound estimate is given
in \cite{Yo2} and in  \cite{MT}.

\begin{proof}[Proof of Theorem C]
Thanks to  Lemma~\ref{LBspBV} proved below, $\B$ satisfies Assumption~B(a) to (f); moreover
by Lemma~\ref{LLYspBV}, we know that condition $S(1)$  is satisfied.
Notice  that all requirements of Assumption T hold, since
part (a), (c) and (d) follow from Assumption T${}'$(a), (c) and (e) directly,
and part (b) follows from the definition of $\wh T$.
Moreover Lemma~\ref{LLYspBV} (iii) gives \eqref{LYineq3}.
Hence Theorem~B can be applied and therefore
conditions $S(2)$ and $S(3)$  are satisfied.

The estimate $\|R_n\|= O(d_n)$ follows from Lemma~\ref{LR_nBV}: we have thus proved the decay of correlations \eqref{fThmA}.

Suppose that $T$ also satisfies \eqref{A.0}; we
denote with $z_n\in I_1$ the point such that $T^n(z_n)=z$.
It is well known that $z_n\sim (\c n)^{-1/\c}$
(see e.g. Lemma 3.1 in \cite{HV}), and then we  obtain
$(T_1^{-n})'(x)=O(n^{-\frac{1}{\c}-1});$ it follows that $d_n=O(n^{-\frac{1}{\c}-1})$.
Since the density function $h$ is bounded on $\wh X$,
$\mu(\tau>k)\le C_1 \nu(\tau>k)\le C_2z_k$ for some $C_1, C_2>0$.
Hence $\disp\sum_{k=n+1}^\infty \mu(\tau>k)=O(n^{-\frac{1}{\c} +1)})$.
\end{proof}

\begin{Lemma}\label{LBspBV}
$\B$ is a Banach space satisfying Assumption B(a) to (f) with
$C_a=C_b=1$.
\end{Lemma}

\begin{proof}
These are standard facts, see for instance the proofs in Chapter 1 in \cite{Br}.
\end{proof}

\begin{Lemma}\label{LLYspBV}
There exist constants $\eta\in (0,1)$ and $D,\bar D>0$ satisfying
\begin{enumerate}
\item[{\rm (i)}] for any $f\in \B$,
$%\begin{eqnarray}\label{fLYineq}
|{\wh \P} f|_{\B}\le \eta|f|_{\B}+D\|f\|_{L^1(\hn)};
$%\end{eqnarray}

\item[{\rm (ii)}] for any $f\in \B$,
%$$%\begin{eqnarray}\label{LYineq2}
%\|R(z)^nf\|_\B\le |z^n|\bigl(\eta^n\|f\|_\B+D_1\|f\|_1\bigr);
%$$%\end{eqnarray}
$%\begin{eqnarray}\label{LYineq2}
\|R(z)f\|_{\B}\le |z|\bigl(\eta\|f\|_{\B}+\bar D\|f\|_{L^1(\hn)}\bigr);
$%\end{eqnarray}
\ and

\item[{\rm (iii)}] for any $f\in \wt\B$,
$%\begin{eqnarray}\label{LYineq2}
\|{\wt \P} \wt f\|_{\wt \B}
\le \eta\|\wt f\|_{\wt \B}+D\|\wt f\|_{L^1({\hn \times \rho})}.
$%\end{eqnarray}
\end{enumerate}
\end{Lemma}

\begin{proof}
(i) Let us denote $x_{ij}=\wh T_{ij}^{-1}(x)$, and
$\wh g(x_{ij})=|\wh T_{ij}'(x_{ij})|^{-1};$  we have
$$
 \wh{\P}f(x)= \sum_{j=1}^K\sum_{i=0}^{\infty}
{f(\wh T_{ij}^{-1}x)}{\wh g(\wh T_{ij}^{-1}x)}\1_{\wh TI_{ij}}(x).
$$

We now take a partition $\xi$ of $\wh TI_{ij}$ into
$\wh T_{ij}a_{ij}=x^{(0)}<x^{(1)}<\cdots <x^{(k_{ij})}=\wh T_{ij}b_{ij}$,
where we assume $\wh T_{ij}a_{ij}< \wh T_{ij}b_{ij}$ without loss
of generality.  Whenever $\wh TI_{ij}$  intersects more
than one intervals $I_k=(a_k,b_k)$ in the case $i=0$, then we
put the endpoints $a_k$ and $b_k$ into the partition.
Denote $x_{ij}^{(\ell)}=\wh T_{ij}^{-1}x^{(\ell)}$.  We have
\begin{equation}\label{fLYineqC1}
\begin{split}
&\sum_{\ell =1}^{k_{ij}}\bigl|f(x_{ij}^{(\ell)})\wh g(x_{ij}^{(\ell)})
                      -f(x_{ij}^{(\ell-1)})\wh g(x_{ij}^{(\ell-1)})\bigr| \\
\le &\sum_{\ell =1}^{k_{ij}}
   \wh g(x_{ij}^{(\ell)})\bigl|f(x_{ij}^{(\ell)})-f(x_{ij}^{(\ell-1)})\bigr|
+\sum_{\ell =1}^{k_{ij}}
\bigl|f(x_{ij}^{(\ell-1)})\bigr|
 \bigl|\wh g(x_{ij}^{(\ell)})-\wh g(x_{ij}^{(\ell-1)})\bigr|.
\end{split}
\end{equation}
By \eqref{fdefdij}, $\wh g(x_{ij}^{(\ell)})\le d_{ij}$ and
by definition
$\sum_{\ell =1}^{k_{ij}}
\bigl|f(x_{ij}^{(\ell-1)})-f(x_{ij}^{(\ell)})\bigr|\le V_{I_{ij}}(f)$.
Also, by the mean value theorem and Assumption T${}'$(d),
$$\disp\frac{|g(\wh x_{ij}^{(\ell)})-\wh g(x_{ij}^{(\ell-1)})|}
{x_{ij}^{(\ell)}-x_{ij}^{(\ell-1)}}\le |\wh g'(c_{ij}^{(\ell)})|
= |\wh T''(c_{ij}^{(\ell)})| / |\wh T'(c_{ij}^{(\ell)})|^2
\le \C,
$$
where $c_{ij}^{(\ell)}\in [x_{ij}^{(\ell-1)}, x_{ij}^{(\ell)}]$.
Using the fact that
$$\lim_{\max\{|x_{ij}^{(\ell)}-x_{ij}^{(\ell-1)}|\}\to 0}\sum_{\ell =1}^{k_{ij}}\
\bigl|f(x_{ij}^{(\ell-1)})\bigr| (x_{ij}^{(\ell)}-x_{ij}^{(\ell-1)})
=\int_{a_{ij}}^{b_{ij}}|f| d\hn,
$$
we get from \eqref{fLYineqC1} that
\begin{equation}\label{festV1}
V_{\wh T I_{ij}}((f\cdot\wh g)\circ \wh T_{ij}^{-1})
\le d_{ij} V_{I_{ij}}(f) +\C \int_{I_{ij}}|f|d\hn.
\end{equation}
Denote $c=\min\{\nu(\wh T I_{ij}): i=1,2,\cdots, 1\le j\le K\}$, where
$c>0$ because there is only a finite number of images $\wh T I_{ij}$.
It can be shown  that (see e.g. \cite{Br}, Ch. 3)
$$
V_{\hat{X}}(\wh\P f)
\le 2\sum_{j=1}^K\sum_{i=0}^{\infty}
  V_{\wh T I_{ij}}((f\cdot\wh g)\circ \wh T_{ij}^{-1})
+ 2c^{-1}\|f\|_1.
$$
By Assumption~T${}'$(c), $d_{ij}\le \D^{-1}$ for all $i=1,2,\cdots$ and
$j=1,\cdots, K$.  Hence
$$
|\wh\P f|_\B=V_{\hat{X}}(\wh\P f)
\le 2\D^{-1} V(f) +2\C \int |f|d\hn +2c^{-1}\|f\|_1
= \eta |f|_\B+D \|f\|_1,
$$
where $\eta=2\D^{-1}<1$ and $D=2\C+2c^{-1}>0$.

Part (ii) and (iii) can be proved similarly to  the proofs
of corresponding part of Lemma~\ref{LLYspV}.
\end{proof}

\begin{Lemma}\label{LR_nBV}
There exists a constant $C_R>0$ such that $\|R_n\|_\B\le C_R d_n$ for all $n>0$.
\end{Lemma}

\begin{proof}
For $f\in \B$, denote
\begin{equation}\label{fdefRijBV}
R_{ij}f=\mathbbold{1}_{\wh X}\cdot \P^i(f\mathbbold{1}_{I_{ij}})(x).
\end{equation}
Hence $\disp R_i=\sum_{j=1}^K R_{ij}$ and
$\disp \wh \P=\sum_{i=0}^\infty\sum_{j=1}^K R_{ij}$
by definition and linearity of $\wh \P$.

Assume $i>0$;
since $\wh T_{ij}[a_{ij}, b_{ij}]=I_0\subset I$, by \eqref{fdefdij},
$\hn(I_{ij})\le d_{ij}\hn(I_0)< d_{ij}$.  Hence, by Assumption~B(b),
\begin{equation}\label{festint}
\int_{I_{ij}}|f|d\hn \le \|f\|_\infty \hn(I_{ij})
\le C_b\|f\|_\B \cdot d_{ij} \hn(I_0)
\le C_b d_{ij}\|f\|_\B.
\end{equation}
Note that $V_{I_{ij}}(f)\le V(f)=|f|_\B<\|f\|_\B$.  By \eqref{festV1},
\begin{equation}\label{festBD}
V_{\wh T I_{ij}}((f\cdot \wh g)\circ \wh T_{ij}^{-1} )
\le d_{ij} \|f\|_\B +\C C_b d_{ij}\|f\|_\B
=(1+\C C_b)d_{ij}\|f\|_\B.
\end{equation}
Since
$R_{ij}f(x)=\1_{\wh X}(x)\cdot (f\cdot \wh g)\circ \wh T_{ij}^{-1} (x),$
we have
$$
|R_{ij}f|_\B
\le 2V_{\wh T I_{ij}}((f\cdot \wh g)\circ \wh T_{ij}^{-1} )
+ 2\frac{1}{\hn(I_0)}\int_{I_{ij}}|f|d\hn.
$$
Moreover by \eqref{festint} and \eqref{festBD},
\begin{equation*}
|R_{ij}f|_\B
\le 2(1+\C C_b)d_{ij}\|f\|_\B + 2C_b d_{ij}\|f\|_\B.
\end{equation*}
On the other hand, by \eqref{fdefRijBV} and \eqref{festint}, we have
$$
\|R_{ij}f\|_{L^1}=\int_{\wh X}\wh \P^{i+1}(f\mathbbold{1}_{I_{ij}})d\hn
= \int_{I_{ij}}fd\hn \le \int_{I_{ij}}|f|d\hn
\le C_b d_{ij}\|f\|_\B.
$$
Hence, we get
$$
\|R_{ij}f\|_\B=|R_{ij}f|_\B +\|R_{ij}f\|_{L^1}
\le [2(1+\C C_b)+3C_b]d_{ij}\|f\|_\B.
$$
By the definition of $R_{ij}$ and $d_n$, we have
$$
\|R_nf\|_\B\le \sum_{j=2}^K\|R_{n-1,j}f\|_\B
\le K'(2+2\C C_b+3C_b)d_{n},
$$
where $K'<K$ is the number of preimages of $I_0$ that are not in $I_1$.
The result follows now with $C_R=K'(2+2\C C_b+3C_b)$.
\end{proof}

%%%%%%%%%%%%%%%%%%%%%%%%%%%%%%%%%%%%%%%%%%%%%%%%%%%%%%%%%%%%%%%
%%%%%%%%%%%%%%%%%%%%%%%%%%%%%%%%%%%%%%%%%%%%%%%%%%%%%%%%%%%%%%%
\section{Multidimensional spaces: generalities
and the role of the derivative}
\label{Sapplicationa}
\setcounter{equation}{0}
%%%%%%%%%%%%%%%%%%%%%%%%%%%%%%%%%%%%%%%%%%%%%%%%%%%%%%%%%%%%%%%
%%%%%%%%%%%%%%%%%%%%%%%%%%%%%%%%%%%%%%%%%%%%%%%%%%%%%%%%%%%%%%%

The main difficulty to investigate the statistical properties
for higher dimensional systems with an indifferent fixed point $p$  is that near $p$ the system could have {{\em unbounded distortion}
in the following sense:
there are uncountably many points $z$ near $p$ such that
for any neighborhood $V$ of $z$, we can find $\hat z\in V$
with the ratio
$$
|\det DT_1^{-n}(z)|/|\det DT_1^{-n}(\hat z)|
$$
unbounded as $n\to \infty$ (see Example in Section 2 in \cite{HV}).
For this reason  we need a more extensive  analysis of the expanding features
around the neutral fixed point which will be  accomplished
by adding  Assumption T\;${}''$ below.

\subsection{Setting and statement of results.}%\label{SSmapft}
Let $X\subset \Bbb R^m$,  $m\ge 1$, be again a compact subset with
$\overline{\intset X}=X$, $d$ the Euclidean distance,
and $\nu$ the Lebesgue measure on $X$ with $\nu (X)=1$.

Assume that $T: X\to X$ is a  map satisfying the following assumptions.

\begin{AssumptionT''}

\begin{enumerate}
\item[{\rm (a)}] {\rm (Piecewise smoothness)}
There are finitely many disjoint open sets
$U_1, \cdots, U_K$ with piecewise smooth boundary
such that $X=\bigcup_{i=1}^K \overline{U_i}$
and for each $i$, $T_i:=T|_{U_i}$ can be extended to a $C^{1+\hat\a}$
diffeomorphism $T_i: \wt {U_i}\to B_{\e_1}(T_iU_i)$,
where $\wt {U_i}\supset U_i$, $\hat\a\in (0,1]$ and $\e_1>0$.

\item[{\rm (b)}] {\rm (Fixed point)}
There is a fixed point $p\in U_1$ such that
$T^{-1}p\notin \partial U_j$ for any $j=1, \dots,K$.

\item[{\rm (c)}] {\rm (Topological mixing)} $T: X\to X$ is topologically mixing.
\end{enumerate}

\end{AssumptionT''}
\begin{Remark}
Assumption T''(b)  allows us to get a good structure for the first return map around any pre-images of $p$ different from $p$ itself. In particular there is an open neighborhood for each of those pre-images which is partitioned in level sets ordered with increasing first return time starting from $2$ and with the same (large) image for the induced map. This induction scheme turns out to be particularly useful when we  consider the transfer operator on the quasi-H\"older function space; in this regard we also refer to our previous paper \cite{HV}.
\end{Remark}
Before continuing with the list of assumptions we need to introduce a few more quantities and notations.\\

For any $\e_0>0$, denote
\begin{eqnarray*}%\label{FdefGUx}
\displaystyle G_U(x, \e,\e_0)
=2\sum_{j=1}^K  \frac{\m(T_{j}^{-1}B_\e(\partial TU_j) \cap B_{(1-s)\e_0}(x))}
 {\m(B_{(1-s)\e_0}(x))}.
\end{eqnarray*}

From now on we assume that the indifferent fixed point $p=0.$\\

For any $x\in U_i$, we define $s(x)$ as the inverse of
the slowest expansion near $x$, that is,
$$
s(x)=\min\bigl\{s: d(x,y)\le sd(Tx, Ty),
 y\in U_i, d(x,y)\le \min\{\e_1, 0.1|x|\} \ \bigr\}.
$$
where the factor $0.1$ forces the points $y$ to stay in a ball around $x$ which does not intersect  the origin,
though any other small factor would work as well.

Take an open neighborhood $Q$ of $p$ such that $TQ\subset U_1$, then let
\begin{equation}\label{fdefs}
s=s(Q)=\max\{s(x): x\in X\backslash Q\}.
\end{equation}

Let $\wh T=\wh T_Q$ be the first return map with respect to
${\wh X}={\wh X}_Q=X\setminus Q$.
Then for any $x\in U_j$, we have $\wh T(x)=T_{j}(x)$
if $T_{j}(x)\notin Q$,
and $\wh T(x)=T_1^iT_{j}(x)$ for some $i> 0$ if $T_{j}(x)\in Q$.
Denote $\wh T_{ij}=T_1^iT_{j}$ for $i\ge 0$.

Further, we take $Q_0=TQ\setminus Q$.
Then we denote $U_{01}=U_1\setminus Q$,
$U_{0j}=U_j\setminus T_j^{-1}Q$ if $j>1$,
and $U_{ij}=\wh T^{-1}_{ij}Q_0$ for $i>0$.
Hence, $\{U_{ij}: i=0,1,2, \cdots\}$ form a partition of $U_j$
for $j=2, \cdots, K$.

For $0<\e\le \e_0$, we denote
\begin{eqnarray*}%\label{FdefGQx}
G_Q(x, \e,\e_0) =2\sum_{j=1}^K \sum_{i=0}^{\infty}
\frac{\m(\hat T_{ij}^{-1}B_\e(\partial Q_0) \cap B_{(1-s)\e_0}(x))}
{\m(B_{(1-s)\e_0}(x))},
\end{eqnarray*}
and
\begin{equation}\label{fdefG}
\displaystyle G(x, \e,\e_0)=G_U(x, \e,\e_0)+G_Q(x, \e,\e_0), \ \
G(\e,\e_0)=\sup_{x\in \wh X}G(x, \e,\e_0).
\end{equation}

\begin{AssumptionT''} {\bf (continued)}
\begin{itemize}
\item[{\rm (d)}] {\rm (Expansion)}
$T$ satisfies: $0<s(x)<1$ $\forall x\in X\setminus\{p\}$.

Moreover, there exists an open region $Q$ with
$p\in Q\subset \overline Q\subset TQ \subset \overline{TQ}\subset U_1$
and constants $\a\in (0,\hat\a]$, $\eta\in (0,1)$, such that
for all $\e_0$ small,
$$
s^\a+ \lambda \le \eta < 1,
$$
where $s$ is defined in \eqref{fdefs} and
\begin{equation}\label{fdeflambda}
\lambda=2\sup_{0<\e\le \e_0}\frac{G(\e,\e_0)}{\e^\a}\e_0^\a.
\end{equation}

\item[{\rm (e)}]  {\rm (Distortion)}
For any $b>0$, there exist $\dd>0$ such that for any small $\e_0$
and $\e\in (0,\e_0)$, we can find $0< N=N(\e)\le \infty$ with
$$
\frac{|\det DT_1^{-n}(y)|}{|\det DT_1^{-n}(x)|}\le 1+\dd\e^\a  \quad
\forall y\in B_\e(x), \ x\in B_{\e_0}(Q_0), \ n\in(0, N_{}],
$$
and
$$
\sum_{n=N_{}}^\infty\sup_{y\in B_\e(x)}|\det DT_1 ^{-n}(y)|
\le b\e^{m+\a}  \quad \forall x\in B_{\e_0}(Q_0),
$$
where $\a$ is given in part (d) and $m$ is the dimension of the ambient space.

\end{itemize}
\end{AssumptionT''}
For sake of simplicity of notations, we may assume $\hat\a=\a$.
\begin{Remark}\label{RGUx}
We stress that the measure $\m(T_{j}^{-1}B_\e(\partial TU_j))$ usually plays
an important role in the study of statistical  properties  of  systems
with discontinuities.
Here $G_U(x, \e,\e_0)$ gives a quantitative measurement of
the competition  between the expansion and the accumulation  of discontinuities near $x$.
We refer to \cite{Ss}, Section 2, for more details about its geometric meaning.
Furthermore it is proved, still  in  \cite{Ss} Lemma 2.1,
that if the boundary of $U_i$ consists of piecewise $C^1$ codimension
one embedded compact submanifolds, then $\disp G_U(\e, \e_0) \le
2N_U\frac{\c_{m-1}}{\c_m}\frac{s\e}{(1-s)\e_0}\bigl(1+o(1)\bigr)$,
where $N_U$ is the maximal number of
smooth components of the boundary of all $U_i$ that meet in one point
and $\c_m$ is the volume of the unit ball in $\mathbb R^m$.
\end{Remark}
\begin{Remark}\label{RG}
If $T^{-1}TQ\cap \partial U_j=\emptyset$ for any $j$,
then for any small $\e_0$, either $G_Q(x, \e,\e_0)=0$ or $G_U(x, \e,\e_0)=0$,
and therefore we have
$G(x, \e,\e_0)=\max\{G_U(x, \e,\e_0),G_Q(x, \e,\e_0)\}$.
\end{Remark}

\begin{Remark}\label{RGQ}
If $T$ has bounded distortion (see below), then $G_Q$ is roughly equal to the ratio
between the volume of $B_{\e_0}(\partial Q_0)$ and the volume of $Q_0$.
Therefore  if $\e_0$ is small enough, then
$\disp \sup_{x\in \wh X}\{G_Q(x, \e,\e_0)\}$ is bounded by
$\disp \sup_{x\in \wh X}\{G_U(x, \e,\e_0)\}$.
\end{Remark}

\begin{Remark}\label{Assmpe1}
We include Assumption T\;${}''$(e) since near the fixed point
the distortion for $DT_1$ is unbounded in general.
It requires that either the distortion of $DT_1^{-n}$ is small,
or $|\det DT_1^{-n}|$ itself is small.
\end{Remark}

\begin{Remark}\label{RAssmpe2}
There are some sufficient conditions under which
Assumption T${\,}''$(d) and (e) could be easily verified.
We refer \cite{HV} for more details, see in particular
 Theorems~B and C in that paper.
\end{Remark}

If near $p$ the distortion is bounded, then Assumption T${}''$(e)
is automatically satisfied and it will be stated as follows
(it could be regarded as the case $N(\e)=\infty$ for any $\e\in (0,\e_0)$):
\begin{AssumptionT''} {\bf (variant)}
\begin{enumerate}
\item[{\rm (e${}'$)}]  {\rm (Bounded distortion)}
There exist $J>0$ such that for any small $\e_0$
and $\e\in (0,\e_0)$,
$$
\frac{|\det DT_1^{-n}(y)|}{|\det DT_1^{-n}(x)|}\le 1+J\e^\a  \quad
\forall y\in B_\e(x), \ x\in B_{\e_0}(Q_0), \ n\ge 0.
$$

\end{enumerate}
\end{AssumptionT''}

\begin{Remark}\label{RAssmpe4}
It is well known that if $\dim X=m=1$, any system that has the form
given by \eqref{B.0} below near the fixed point,
satisfies Assumption~T${}''$(e${}'$).
The systems given in Example~\ref{example0}
satisfy it too.
\end{Remark}

To estimate the decay rates, we often consider the following special cases:
there are constants $\c'>\c>0$, $C_i, C_i'>0$, $i=0,1,2$,
such that in a neighborhood of the
indifferent fixed point $p=0$:
\begin{equation}\label{B.0}
\begin{split}
    |x|\bigl(1-C_0'|x|^\c+O(|x|^{\c'})\bigr)
\le  &|T_1^{-1}x| \le
    |x|\bigl(1-C_0|x|^\c+O(|x|^{\c'})\bigr), \\
1-C_1' |x|^\c +O(|x|^{\c'})\le  &\|DT_1^{-1}(x)\|
    \le  1-C_1 |x|^\c+O(|x|^{\c'}),  \\
 C_2' |x|^{\c-1}+O(|x|^{\c'-1}) \le &\|D^2T_1^{-1}(x)\|
 \le  C_2  |x|^{\c-1}+O(|x|^{\c'-1}).
\end{split}
\end{equation}
where $||DT^{-1}_1||, ||DT||$ etc., denote the operator norms.

We  now define the space of functions particularly adapted
to study the action of the transfer operator on the class of maps
just introduced.  If $\Omega$ is a Borel subset of $\wh X$,
we define the oscillation of $f$ over
$\Omega$ by the difference of essential supremum and essential
infimum of $f$ over $\Omega$:
$$
\osc(f,\Omega) = \Esup_\Omega f - \Einf_\Omega f.
$$
We notice that the function $x\rightarrow
\mbox{osc}(f, \ B_{\epsilon}(x))$ is measurable.

For $0<\a <1$ and $\e_0>0$, we define the quasi-H\"older seminorm of $f$
with $\supp f\subset \wh X$  as\footnote{Since the boundary of $\wh X$ is
piecewise smooth, we could define the space of the function directly on
$\wh X$ instead of ${\mathbb R}^m$ as it was done in \cite{Ss}.}
\begin{equation}\label{fseminorm''}
|f|_{\B} = \sup_{0<\epsilon\le\epsilon_0}
\epsilon^{-\alpha} \int_{\wh X} \mbox{osc}(f, B_{\epsilon}(x)) d\hn(x),
\end{equation}
where $\hn$ is the normalized Lebsegue measure on $\wh X$,
and we take the space of  functions as
\begin{eqnarray}\label{fdefB}
{\B} =  \left\{ f\in L^1(\wh X,\hn): |f|_{\B} <\infty \right\},
\end{eqnarray}
and then equip it with the norm
\begin{equation}\label{fnorm''}
\|\cdot \|_{\B}= \|\cdot \|_{L^1(\wh X,\hn)}+|\cdot |_{\B}.
\end{equation}
Clearly, the space ${\B}$ does not depend
on the choice of $\e_0$, though $|\cdot|_{\B}$ does.

Let $s_{ij}=\sup\bigl\{\|D\wh T^{-1}_{ij}(x)\|: x\in B_{\e_0}(Q_0)\bigr\}$,
and
$s_n=\max\bigl\{s_{n-1,j}: j=2,\cdots, K\bigr\}$.

\begin{TheoremD}\label{ThmD}
Let $\wh X$, $\wh T$ and $\B$ be defined as above.
Suppose $T$ satisfies Assumption T$\,{}''$(a) to (e).
Then there exist $\e_0\ge \e_1>0$ such that Assumption B(a) to (f)
and conditions $S(1)$ to $S(3)$  are satisfied and
$\|R_n\|= O(s_n^{\a})$.
Hence, if $\sum_{k=n+1}^\infty s_n^{\a} = O(n^{-\b})$ for some $\b>1$,
then there exists $C>0$ such that for any functions
$f\in \B$, $g\in L^\infty(X,\nu)$ with
$\supp{f}, \; \supp{g}\subset {\wh X}$, (\ref{fThmA}) holds.
\end{TheoremD}

Before giving the proof, we present an example.

\begin{Example}\label{example0}
Assume that $T$ satisfies Assumption T${}''$(a) to (d),
and near the fixed point $p=0$, the map $T$ satisfies
$$
T(z)= z ( 1+|z|^\c+O(|z|^{\c'})),
$$
where $z\in X\subset {\mathbb R}^m$ and $\c'>\c$.
\end{Example}

Denote $z_n=T_1^{-n}z;$ we showed in Lemma 3.1 in \cite{HV} that
$\disp |z_n|= \frac{1}{(\c n)^{1/\c}}+O\Bigl(\frac{1}{n^{1/{\bar \c}}}\Bigr)$,
where $\bar \c<\c$.
Using this fact we can check that $T$  satisfies also
Assumption~T${}''$(e$\,{}')$; hence, the theorem can be applied.

If the dimension $m=1$, then $T^n$ maps the interval
$[z_{n+1}, z_n]=[z_{n+1}, T(z_{n+1})]$ to its image $[z_1, z_0]$ bijectively.
It follows that $\|DT^{-n}_1\|$ is roughly proportional
to $|z_n|^{1+\c}/(|z_0|-|z_1|)$,
since the length of the interval $[z_{n+1}, T(z_{n+1})]$ is roughly
equal to $|T(z_{n+1})-z_{n+1}|\sim |z_n|^{1+\c}$, see also Lemma 3.1 and Lemma 3.2 in \cite{HV} for  a more formal derivation.
So $\disp s_n=O\Bigl(\frac{1}{n^{1+1/\c}}\Bigr)$ and
$\disp \sum_{k=n+1}^\infty s_k^{\a}= O\Bigl(\frac{1}{n^{\frac{\a}{\c}+\a-1}}\Bigr)$.
If $\c\in (0,1)$ is such that $\a(1/\c+1)>1$, the series is convergent.
Also, as stated in Theorem~C in the last section,
$\disp \sum_{k=n+1}^\infty \mu(\tau>k)=O\Bigl(\frac{1}{n^{\frac{1}{\c}-1}}\Bigr)$.
So if $\a(1/\gamma+1)>1/\gamma$,
the sum involving $s^\a_k$ decreases faster.
We get that the decay rate is given by
$$
\Bigl| \Cov(f, g\circ T^n)\Bigr|
= O\Bigl(\sum_{k=n+1}^\infty \mu(\tau>k)\Bigr)
= O\Bigl(\frac{1}{n^{\b-1}}\Bigr),
$$
for $f\in \B$, $g\in L^\infty(X,\nu)$
with $\supp{f}, \; \supp{g}\subset {\wh X}$ and with $\beta=\frac{1}{\c}.$
This gives the same results as in Theorem~C for quasi H\"older test functions
instead that for functions of bounded variation.

On the other hand, if $m\ge 2$, then $T_1^{-n}$ maps a sphere
about the fixed point of radius $|z|$ to a sphere of radius $|z_n|$,
if higher order terms are ignored.
Hence, $DT_1^{-n}$ contracts vectors in the tangent space of the sphere
at the rate of order $|z_n|$.
To see the contracting rates along the radial direction,
i.e., the direction orthogonal to the tangent space of the spheres,
we note that restricted to each ray the map has the form
$T(r)=r(1+r^\c +O(r^{\c'}))$.
Hence, by the above arguments for one dimensional case, $DT_1^{-n}$
contracts vectors in the radial direction at the rate of order $|z_n|^{1+\c}$.
Therefore the norm $\|DT_1^{-n}\|$ is roughly proportional to $|z_n|$, and
$\disp s_n=O\Bigl(\frac{1}{n^{1/\c}}\Bigr)$ and
$\disp \sum_{k=n+1}^\infty s_k^{\a}= O\Bigl(\frac{1}{n^{(\a/\c)-1}}\Bigr)$.
If $\c\in (0,1/2)$ is such that $\a/\c>1$, the series is convergent. By defining $\beta:=\frac{\a}{\c}-1$ we can now consider the three cases $\beta>2, 1<\beta<2, \beta=2$ in order to determine the error term $F_{\beta}(n).$ Let us take, for instance, $\beta>2$, which requires $\a/\c>3.$

Note that $\nu(\tau>n)$ is of the same order as $|z_n|^m$, and
therefore  $\disp \mu(\tau>n)=O\Bigl(\frac{1}{n^{m/\c}}\Bigr)$.
It follows that
$\disp \sum_{k=n+1}^\infty \mu(\tau>k)=O\Bigl(\frac{1}{n^{(m/\c)-1}}\Bigr)$.
Since the order is higher, by \eqref{fThmA}, we get $
\Bigl| \Cov(f, g\circ T^n)\Bigr|\le C/n^{\beta}.$

%%%%%%%%%%%%%%%%%%%%%%%%%%%%%%%%%%%%%%%%%%%%%%%%%%%%%%%%%%%%%%%%
%%%%%%%%%%%%%%%%%%%%%%%%%%%%%%%%%%%%%%%%%%%%%%%%%%%%%%%%%%%%%%%%
\subsection{Proof of Theorem D}%\label{SSmapft}
%%%%%%%%%%%%%%%%%%%%%%%%%%%%%%%%%%%%%%%%%%%%%%%%%%%%%%%%%%%%%%%%
%%%%%%%%%%%%%%%%%%%%%%%%%%%%%%%%%%%%%%%%%%%%%%%%%%%%%%%%%%%%%%%%
The proof of Theorem D requires a few preparatory lemmas.

First of all and in order to deduce the spectral properties of $\hat{\mathcal{P}}$
from the Lasota-Yorke inequality, one needs to verify Assumption B
on the space of functions $\B$.

\begin{Lemma}\label{LBsp2}
$\B$ is a Banach space satisfying Assuptions B(a) to (f) with
 $C_a=2C_b=2\gamma_m^{-1}\epsilon_0^{-m}$, where $\gamma_m$
is the volume of the unit ball in $\mathbb{R}^m$.
\end{Lemma}

\begin{proof}
Parts (a), (b) and (c) are stated in Propositions 3.3 and 3.4 in \cite{Ss}
with $C_b=\max\{1,\e^\a\}/\c_m \e_0^m$ and $C_a=2\max\{1,\e^\a\}/\c_m \e_0^m$.
%for the space ${\B_H}$ and immediately they transfer to $\B$.
Part (d)  follows from the fact that H\"older continuous functions
with compact support in $\wh X$ are dense in $L^1(\wh X, \hn)$.

Let us now assume $f(u)=\lim_{n\to\infty} f_n(u)$ for $\hn$-a.e.
$u\in \mathbb{R}^m$.
Take $x\in \mathbb{R}^m$, and $\e\in (0,\e_0)$.  It is easy to see that for
almost every pair of $y,z\in B_\e(x)$, we have
$$
|f(y)-f(z)|\le \lim_{n\to\infty}|f_n(y)-f_n(z)|
\le \liminf_{n\to\infty}\osc(f_n, B_\e(x)).
$$
Hence, $\osc(f, B_\e(x))\le \liminf_{n\to\infty}\osc(f_n, B_\e(x))$.
By Fatou's lemma, we have
$$\int \osc(f, B_\e(x))d\hn
\le \liminf_{n\to\infty}\int \osc(f_n, B_\e(x))d\hn.
$$
This implies $|f|_\B\le \liminf_{n\to\infty}|f_n|_\B$.
We get part (e).

It leaves to show part (f).  For a function $f\in \B$, denote
$$
{\mathcal D}_n(f)
=\Bigl\{x\in \mathbb{R}^m: \liminf_{\e\to 0}\osc(f, B_{\e}(x)) > \frac 1n\Bigr\},
\quad
{\mathcal D}(f)=\bigcup_{n=1}^\infty {\mathcal D}_n(f).
$$
Clearly ${\mathcal D}(f)$ is the set of discontinuity points of $f$.
If $\hn( {\mathcal D}(f))>0$, then there exists $N>0$ such that
$\leb ({\mathcal D}_N(f))>\iota>0$.
Notice that ${\mathcal D}_N(f)=\bigcup_{k\ge 1} S_k$,
where $S_k=\bigcap_{n\ge k}\{x: \osc(f, B_{\frac1n}(x))>\frac1N\}$
is an increasing sequence of measurable sets.

For $k$ big enough we still have $\hn(S_k)>\iota$ and therefore,
for such a $k$:
$$
|f|_\B
\ge \sup_{\e>0} \e^{-a} \int_{{\mathcal D}_N(f)} \osc(f, B_{\e}(x))d\hn(x)
\ge \sup_{\e>0} \e^{-a} \int_{S_k} \osc(f, B_{\e}(x))d\hn(x)
=\infty.
$$
This means  $f\notin \B$; in other words, any $f\in \B$ satisfies
$\hn({\mathcal D}(f))=0$.

Take any $f\in \B$ with $f\ge 0$ almost everywhere.
If $f(x)=2c>0$ for some $x\notin {\mathcal D}(f)$, then
there is $\e>0$ such that $\osc(f, B_{\e}(x))\le c$.
Hence, $f(x')\ge c>0$ for almost every point $x'\in B_{\e}(x)$.
So $B_{\e}(x)\setminus \{f>0\}$ has Lebesgue measure zero.
This implies that $\{f>0\}$ is almost open and therefore part (f) follows.
\end{proof}
Before stating the next lemma, we recall that the space $\B$ depends on the exponent $\alpha$ and the value of the seminorms on $\epsilon_0$: as we did above, we will not index $\B$ with these two parameters. Moreover all the integrals in the next proof will be performed over $\wh X$.
\begin{Lemma}\label{LLYspV}
There exists $\e_*>0$ such that for any $\e_0\in (0, \e_*)$, we can find
constants $\eta\in (0,1)$ and $D, \hat D>0$
%, where $\eta$ is decreasing with $\e_0$,
satisfying
\begin{enumerate}
\item[{\rm (i)}] for any $f\in \B$,
$%\begin{eqnarray}\label{fLYineq}
|{\wh \P} f|_{\B}\le \eta|f|_{\B}+D\|f\|_{L^1(\hn)};
$%\end{eqnarray}

\item[{\rm (ii)}] for any $f\in \B$,
%$$%\begin{eqnarray}\label{LYineq2}
%\|R(z)^nf\|_\B\le |z^n|\bigl(\eta^n\|f\|_\B+D_1\|f\|_{L^1(\nu)}\bigr);
%$$%\end{eqnarray}
$%\begin{eqnarray}\label{LYineq2}
\|R(z)f\|_{\B}
\le |z|\bigl(\eta\|f\|_{\B}+\hat D\|f\|_{L^1(\hn)}\bigr); \; and
$%\end{eqnarray}

\item[{\rm (iii)}] for any $\wt f\in \wt\B$,
$%\begin{eqnarray}\label{LYineq2}
\|{\wt \P} \wt f\|_{\wt \B}
\le \eta\|\wt f\|_{\wt \B}+D\|\wt f\|_{L^1({\hn \times \rho})}.
$%\end{eqnarray}
\end{enumerate}
\end{Lemma}

\begin{proof}
%Note that by Assumption T${}''$ (d), $s<1$.  So $T$ is uniformly
%expanding on $\wh X$.  A standard argument gives that
%if $f\in \H^\a(V_{ij}, H)$ for some large $H$, then
%${\wh P}f\in \H^\a(V_{i-1,j}, H)$.
%Hence, $f|_\V\in \H^\a$ implies for ${\wh P}f|_\V\in \H^\a$.
%Also, if we can choose $\e_0>0$ such that part (1)
%of the conclusion holds, then we get that $|f|_\B<\infty$
%implies $|{\wh P}f|_\B<\infty$.
%By definition we have $\wh \P(\B)\subset \B$.

By Assumption T${}''$ (d), $s^\a+\l<1$.
Therefore if we first choose $b$ small enough, we obtain $\dd$
according to Assumption T${}''$(e), and then we can take $\e_0$
small enough in order to get
\begin{eqnarray}\label{fdefeta}
\eta:=(1+\dd \e^\a_0)(s^\a+\l)+2\c_m^{-1}bK'<1,
\end{eqnarray}
where $K'$ is the number of $j$ such that $U_{ij}\not=\emptyset$.
Clearly, $\eta$ is decreasing with $\e_0$.
Let us define:
\begin{eqnarray}\label{fdefD}
D:=2\dd+2(1+\dd \e^\a_0)\l/\e_0^\a+ 2\c_m^{-1}bK'>0.
\end{eqnarray}

For any $x\in \wh X$, let us denote $x_{ij}=\wh T_{ij}^{-1}x$,  $\wh g_{ij}(x)=|\det D\wh T_{ij}(x)|^{-1}$ and
for $f\in \B$:
\begin{equation}\label{fdefRij}
R_{ij}f=\mathbbold{1}_{\wh X}\cdot \P^i(f\mathbbold{1}_{U_{ij}})(x).
\end{equation}
Clearly,
\begin{equation}\label{fexpR}
R_{ij}f(x)=f(x_{ij})\hat g(x_{ij})\1_{U_{ij}}(x_{ij}).
\end{equation}
Hence $R_i=\sum_{j=1}^K R_{ij}$ and $\wh \P=\sum_{i=0}^\infty\sum_{j=1}^K R_{ij}$
by definition and the linearity of $\wh \P$.
We also define
\begin{eqnarray*}
G_{ij}(x, \e,\e_0)
=2\frac{\m(\wh T_{ij}^{-1}B_\e(\partial \wh TU_{ij}) \cap B_{(1-s)\e_0}(x))}
 {\m(B_{(1-s)\e_0}(x))}. %\quad
%G_{ij}(\e,\e_0)=\sup_{x\in \wh X}G_{ij}(x, \e,\e_0)
\end{eqnarray*}
Clearly,
$G(x,\e,\e_0)=2\sum_{i=0}^{\infty} \sum_{j=1}^K G_{ij}(G(x,\e,\e_0))$.

For any $\e\in (0,\e_0]$, take $N=N(\e)>0$ as in Assumption T${}''$(e).

For $i\le N(\e)$ and  by the proof of Proposition 6.2 in \cite{HV}, we know that
\begin{equation}\label{fosc}
\begin{split}
&\osc(R_{ij}f, B_\e(x)\bigr)
=\osc\bigl((f {\wh g})\circ \wh T_{ij}^{-1}\mathbbold{1}_{\wh TU_{ij}}, \
           B_\e(x)\bigr) \\
=&\osc\bigl((f {\wh g})\circ \wh T_{ij}^{-1},  B_\e(x)\bigr)
        \mathbbold{1}_{\wh TU_{ij}}\!(x)
  +\bigl[2\Esup_{B_\e(x)}(f {\wh g})\circ \wh T_{ij}^{-1} \bigr]
        \mathbbold{1}_{B_\e(\partial \wh TU_{ij})}(x).   \\
%=:&\!\!R^{(1)}_{ij}(x)\mathbbold{1}_{\wh TU_{ij}}(x)
% +R^{(2)}_{ij}(x)\mathbbold{1}_{B_\e(\partial \wh TU_{ij})}(x).
\end{split}
\end{equation}
The computation in that proof also gives
\begin{equation*}
\begin{split}
&\osc\bigl(f {\wh g}, \ \wh T_{ij}^{-1} B_\e(x)\cap U_{ij}\bigr) \\
\le& (1+\dd \e^\a)\osc\bigl(f, \ B_{s\e}(x_{ij}) \cap U_{ij}\bigr)
     {\wh g}(x_{ij})
  +2\dd \e^\a |f(x_{ij})|{\wh g}(x_{ij}).
\end{split}
\end{equation*}
Notice that $\osc\bigl(f, \ B_{s\e}(x_{ij}) \cap U_{ij}\bigr)
\le\osc\bigl(f, \ B_{s\e}(x_{ij})\bigr)$.
By integrating  and using \eqref{fexpR} we get
\begin{equation}\label{fR^1}
\begin{split}
&  \int \osc\bigl((f {\wh g})\circ \wh T_{ij}^{-1},  B_\e(\cdot)\bigr)
          \mathbbold{1}_{\wh TU_{ij}} d\hn  \\
\le& \int \bigl[
 (1+\dd \e^\a)R_{ij}\osc\bigl(f, \ B_{s\e}(\cdot)\bigr)
  +2\dd \e^\a R_{ij}|f|\bigr]d\hn.
\end{split}
\end{equation}
On the other hand, by the same arguments as in Section 4 of \cite{Ss},
we  get
\begin{equation}\label{fR^2}
\begin{split}
&\int 2\bigl[\Esup_{B_{s\e}(x)}(f {\wh g})\circ \wh T_{ij}^{-1} \bigr]
        \mathbbold{1}_{B_\e(\partial \wh TU_{ij})}(x) d\hn     \\
%\le &\int 2 (1+\dd \e^\a)\bigl[\Esup_{B_{s\e}(x_{ij})}|f|\bigr]
%        \mathbbold{1}_{B_\e(\partial \wh TU_{ij})}(\wh Tx_{ij}) d\m  \\
\le&  2(1+\dd \e^\a)\int_{\wh X} G_{ij}(x, \e, \e_0)
 \bigl[|f|(x) + \osc(f, B_{\e_0}(x))\bigr] d\hn.
\end{split}
\end{equation}
Therefore by \eqref{fosc}, \eqref{fR^1} and \eqref{fR^2},
\begin{equation}\label{fRija}
\begin{split}
& |R_{ij}f|_\B=\sup_{0<\e\le \e_0}\e^{-\a}
   \int \osc(R_{ij}f, B_\e(\cdot)\bigr)d\hn   \\
\le&  \sup_{0<\e\le \e_0}\e^{-\a}\int \bigl[
 (1+\dd \e^\a)R_{ij}\osc\bigl(f, \ B_{s\e}(\cdot)\bigr)
   +2\dd \e^\a R_{ij}|f|\bigr]d\hn    \\
+&\sup_{0<\e\le \e_0}\e^{-\a}2 (1+\dd \e^\a)\int_{\wh X} G_{ij}(x, \e, \e_0)
 \bigl[|f|(x) + \osc(f, B_{\e_0}(x))\bigr] d\hn.
\end{split}
\end{equation}

For $i>N(\e)$, by the  definition of oscillation we obtain directly that
\begin{eqnarray*}%\label{fR_ijb}
 \osc(R_{ij}f, B_\e(x)\bigr)
\le 2\|f\|_\infty \sup_{\wh T_{ij}^{-1}B_{\e}(x)} {\wh g}.
\end{eqnarray*}
Hence, by Assumption B(b) with $C_b=\c_m^{-1}\e_0^{-m}$, we have
\begin{equation}\label{fRijb}
\begin{split}
& |R_{ij}f|_\B=\sup_{0<\e\le \e_0}\e^{-\a}
   \int \osc(R_{ij}f, B_\e(\cdot)\bigr)d\hn   \\
\le& 2\|f\|_\infty
\sup_{0<\e\le \e_0} \e^{-\a}\int \sup_{\wh T_{ij}^{-1}B_{\e}(x)} {\wh g}\ d\hn \\
\le &2(\c_m\e_0^m)^{-1}(|f|_\B+\|f\|_\1)\;
\sup_{0<\e\le \e_0}\e^{-\a} \int \sup_{\wh T_{ij}^{-1}B_{\e}(x)} {\wh g}\ d\hn.
\end{split}
\end{equation}

(i) We first note that for all ${0<\e\le \e_0}$,
\begin{equation}\label{fsum1}
\begin{split}
 \e^{-\a}  \sum_{i=0}^{N(\e)}\sum_{j=1}^K
  \int &R_{ij}\osc\bigl(f, \ B_{s\e}(\cdot)\bigr)d\hn
\le \e^{-\a}
    \int {\wh \P}\osc\bigl(f, \ B_{s\e}(\cdot)\bigr)d\hn \\
\le&  s^{\a}(s\e)^{-\a}
   \int \osc\bigl(f, \ B_{s\e}(\cdot)\bigr)d\hn
\le s^{\a} |f|_\B,
\end{split}
\end{equation}
\begin{equation}\label{fsum2}
\begin{split}
& \e^{-\a} \sum_{i=0}^{N(\e)}\sum_{j=1}^K \int
      2 (1+\dd \e^\a)G_{ij}(\cdot, \e, \e_0)
 \bigl[|f| + \osc(f, B_{\e_0}(\cdot))\bigr] d\hn  \\
\le& \e^{-\a} 2 (1+\dd \e^\a)G(\e, \e_0)\int
   \bigl[|f| + \osc(f, B_{\e_0}(\cdot))\bigr] d\hn \\
\le &(1+\dd \e^\a)\lambda \bigl[\e_0^{-\a}\|f\|_{1} + |f|_\B\bigr],
\end{split}
\end{equation}
where we used \eqref{fdefG} and \eqref{fdeflambda}.
Also, by Assumption T${}''$(e) and Assumption B(b) with
$C_b=\c_m^{-1}\e_0^{-m+\a}$, we have
that for all ${0<\e\le \e_0}$:
\begin{equation}\label{fsum3}
 \e^{-\a} \|f\|_\infty \int
  \sum_{N(\e)}^\infty\sum_{j=1}^{K'}
  \sup_{\wh T_{ij}^{-1}B_{\e}(x)} {\wh g} \ d\hn
\le \e^{-\a} \|f\|_\infty \cdot bK'\e^{m+\a}
\le \c_m^{-1}bK'\|f\|_\B.
\end{equation}

Since $\wh \P f(x)=\sum_{i=0}^\infty \sum_{j=1}^K R_{ij}f(x)$,
by \eqref{fRija} and \eqref{fRijb},
and using (\ref{fsum1}) to (\ref{fsum3}), we obtain that $|\wh \P f|_\B$
is bounded by
\begin{equation*}%\label{fsum3}
\begin{split}
& \sup_{0<\e\le \e_0}\e^{-\a}
 \Bigl[\int \sum_{i=0}^\infty\sum_{j=1}^K \osc(R_{ij}f, B_\e(x)) d\hn
+  \int  \sum_{i=0}^\infty\sum_{j=1}^K \osc(R_{ij}f, B_\e(x))d\hn\Bigr] \\
\le& (1+\dd \e^\a_0)s^\a|f|_\B + 2\dd\|f\|_1
  +(1+\dd \e^\a_0)\l (\e^{-\a}_0\|f\|_1 + |f|_\B )
  + 2\c_m^{-1}bK'\|f\|_\B                \\
\le&  [(1+\dd \e^\a_0)(s^\a+\l)+2\c_m^{-1}bK']|f|_\B
 +[2\dd+2(1+\dd \e^\a_0)\l/\e_0^\a+ 2\c_m^{-1}bK']\|f\|_1.
\end{split}
\end{equation*}
By definition of $\eta$ in \eqref{fdefeta} and
$D$ in \eqref{fdefD} we get the desired inequality.

(ii) We begin to  note that for any real valued function $f$ and $z\in {\mathbb C}$, we have
$\osc(zf, B_\e(x))=|z|\osc(f, B_\e(x))$.
Moreover we point out  that if $\{a_n\}$ is a sequence of positive numbers
and $z\in \overline{\mathbb D}$, then
$|\sum_{n=1}^\infty z^n a_n|\le |z|\sum_{n=1}^\infty a_n$.
Hence we have
\begin{equation*}%\label{fsum1}
\begin{split}
|R(z)f|_\B
\le  |z| \sup_{0<\e\le \e_0}\e^{-\a}\sum_{i=0}^\infty \sum_{j=1}^K
   \int\osc(R_{ij}f, B_{\e}(x))d\hn
\le |z| |\wh \P f|_\B.
\end{split}
\end{equation*}
By part (i), the inequality becomes
$$
|R(z)f|_\B \le |z|(\eta |f|_\B+D\|f\|_1).
$$
Since $\wh\P$ and $R_n$ are positive operators, we get
\begin{equation*}%\label{fsum1}
\bigl\|R(z)f\|_1
\le \sum_{n=1}^\infty  \bigl\|z^n R_nf\bigr\|_1
\le |z|\sum_{n=1}^\infty \bigl\|R_{n}|f| \bigr\|_1
=  |z|\bigl\|\wh\P |f|\bigr\|_1= |z|\bigl\|f\bigr\|_1,
\end{equation*}
from which
$$
\|R(z)f\|_\B \le |z|(\eta \|f\|_\B+(D+1)\|f\|_1).
$$
%Using induction on $n$,
We finally get the expected result with
$\hat D=D+1$.

(iii) The transfer operator $\wt \P$ has the form (see also \cite{ADSZ})
\begin{eqnarray*}%\label{fosct1}
({\wt \P}\wt f)(x,y)
=\sum_{n=0}^\infty \sum_{j=1}^K
\wt f(\wh T_{ij}^{-1}x, S(U_{ij})^{-1}(y))g(\wh T_{ij}^{-1}x)
\1_{\wh TU_{ij}}(x,y),
\end{eqnarray*}
for any $\wt f\in \wt \B$,
where $S(U_{ij}): Y\to Y$ are automorphisms.
Let us denote:
$$
(\wt R_{ij}\wt f)(x, y)
=\wt f(\wh T_{ij}^{-1}x, S(U_{ij})^{-1}(y))g(\wh T_{ij}^{-1}x)
\1_{\wh TU_{ij}}(x,y).
$$

Following the same computations as above, we  get formulas similar
to \eqref{fRija} and \eqref{fRijb} but with
$R_n$ and $\wh T_{ij}$ replaced by $\wt R_n$ and $\wt T_{ij}$ respectively,
and $f(\cdot)$ replaced by $\wt f(\cdot, y)$.
Denote $y_1=S(U_{ij})^{-1}(y)$;
instead of \eqref{fRija} and \eqref{fRijb}, we get that for $i< N(\e)$,
\begin{equation*}%\label{ftR_ija}
\begin{split}
&|\wt R_{ij}\wt f(\cdot, y)|_{\B}=\sup_{0<\e\le \e_0}\e^{-\a}
   \int \osc(\wt R_{ij}\wt f(\cdot, y_1), B_\e(\cdot)\bigr)d\hn   \\
\le&  \sup_{0<\e\le \e_0}\e^{-\a}\int \Bigl[\Bigl(
 (1+\dd \e^\a)\wt R_{ij}\osc\bigl(\wt f(\cdot, y_1), \ B_{s\e}(\cdot)\bigr)
   +2\dd \e^\a \wt R_{ij}|\wt f(\cdot, y_1)|\Bigr)    \\
+&2G_{ij}(x, \e, \e_0)(1+\dd \e^\a)
   \Bigl(\osc(\wt f(\cdot, y_1), B_{\e}(\cdot))
  + |\wt f(\cdot, y_1)|\Bigr) \Bigr]d\hn,
\end{split}
\end{equation*}
and for $i\ge N(\e)$,
\begin{equation*}%\label{ftR_ijb}
\begin{split}
&|\wt R_{ij}\wt f(\cdot, y)|_{\B}=\sup_{0<\e\le \e_0}\e^{-\a}
   \int \osc(\wt R_{ij}\wt f(\cdot, y_1), B_\e(\cdot)\bigr)d\hn   \\
\le& 2(\c_m\e_0^m)^{-1}(|\wt f(\cdot, y_1)|_\B
+\|(\wt f\cdot, y_1)\|_{L^1(\nu)})\e^{-\a}
\sup_{0<\e\le \e_0}\int \sup_{\wh T_{ij}^{-1}B_{\e}(x)} {\wh g}d\hn.
\end{split}
\end{equation*}

We observe that for any $x$, $S(U_{ij}): Y\to Y$ preserves the measure $\rho$;  we set
$$
\bar f(x)=\int_{\mathbb S}\wt f(x,y_1)d\rho(y), \quad
\overline{\osc}\bigl(\wt f(\cdot), B_\e(\cdot)\bigr)
=\int_{\mathbb S}\osc\bigl(\wt f(\cdot, y_1), B_\e(\cdot)\bigr)d\rho(y).
$$
By integrating  with respect to $y$, and using
Fubini's theorem, we get
\begin{equation*}%\label{ftR_ija}
\begin{split}
|\wt R_{ij}\wt f|_{\wt \B}
\le & \sup_{0<\e\le \e_0}\e^{-\a}\int \Bigl[\Bigl((1+\dd \e^\a)
 \wt R_{ij}\overline\osc\bigl(\wt f(\cdot), \ B_{s\e}(\cdot)\bigr)
   +2\dd \e^\a \wt R_{ij}|\bar f(\cdot)|\Bigr)    \\
+&2G_{ij}(x_{ij}, \e, \e_0)(1+\dd \e^\a)
   \Bigl(\overline\osc(\wt f(\cdot), B_{\e}(\cdot))
  + |\bar f(\cdot)| \Bigr) \Bigr]d\hn
\end{split}
\end{equation*}
and
\begin{equation*}%\label{ftR_ijb}
|\wt R_{ij}\wt f|_{\wt \B}
\le 2(\c_m\e_0^m)^{-1}(|\wt f|_{\wt \B}+\|\wt f\|_{L^1(\hn\times \rho)})\;\e^{-\a}
\sup_{0<\e\le \e_0}\int \sup_{\wh T_{ij}^{-1}B_{\e}(x)} {\wh g}d\hn.
\end{equation*}

Using Fubini's theorem again, we  also have
$\disp |\wt f|_{\wt\B}
=\sup_{0<\e\le \e_0}\e^{-\a}\int \overline\osc(\wt f(\cdot), B_{\e}(\cdot))d\hn$,
and $\disp |\wt f|_{L^1({\hn \times \rho})}=\int |\bar f(\cdot)| d\hn$.
Using the same arguments as in the proof of part~(i) we get
\begin{equation*}%\label{fosc1}
\begin{split}
&|{\wt P}\wt f(\cdot, y)|_{\wt\B}
\le \sum_{n=0}^\infty \sum_{j=1}^K |\wt R_{ij}\wt f|_{\wt\B}
\le (1+\dd \e^\a_0)s^\a|\wt f|_{\wt\B} +
        2\dd\|\wt f\|_{L^1({\hn \times \rho})} \\
+&(1+\dd \e^\a_0)\l \bigl(|\wt f|_{\wt\B}
         + \e^{-\a}_0\|\wt f\|_{L^1({\hn \times \rho})}\bigl)
+ 2\c_m^{-1}bK'
     \bigl(|\wt f|_{\wt\B} + \|\wt f\|_{L^1({\hn \times \rho})}\bigl),
\end{split}
\end{equation*}
and therefore the result of part (iii)
with the same $\eta$ and $D$ given in \eqref{fdefeta} and \eqref{fdefD}
respectively.
\end{proof}

\begin{Lemma}\label{LR_ns}
There exists a constant $C_R>0$ such that $\|R_n\|_\B\le C_R s_n^{\a}$
for all $n>0$.
\end{Lemma}

\begin{proof}
Since $R_i=\sum_j R_{ij}$, we only need to prove the results for $R_{ij}$.

Let us take $\e\in (0, \e_0],$ choose any $b>0$ and
let $N(\e)$ be given by Assumption~T${}''$(e).

We first consider the case $n=i+1\le N(\e)$.

By the definition of $R_{ij}$ given in \eqref{fdefRij},
we have for any $f\in\B$,
\begin{equation}\label{fRijint}
\int R_{ij}f d\hn
=\int \mathbbold{1}_{\wh X}
   \cdot \P^{i+1}(f\mathbbold{1}_{U_{ij}})d\hn
= \int_{\wh X} f\mathbbold{1}_{U_{ij}}d\hn
=\int_{U_{ij}} fd\hn.
\end{equation}
We now denote
$d_{ij}=\sup\bigl\{|\det D\wh T^{-1}_{ij}(x)|: x\in B_\e(Q_0)\bigr\}$.
Since for any $x$, $|\det D\wh T^{-1}_{ij}(x)|\le\|D\wh T^{-1}_{ij}(x)\|$,
we have $d_{ij}\le s_{ij}$.  Since $\wh TU_{ij}=Q_0$,
\begin{equation}\label{fUijb}
\nu(U_{ij})\le d_{ij}\nu(Q_0)\le s_{ij}\nu(Q_0).
\end{equation}
Hence by Assumption~B(b),
\begin{equation}\label{fRijL1}
\int R_{ij}f d\hn
\le \|f\|_{L^\infty(\hn)} \nu(U_{ij})
\le C_b \nu(Q_0) s_{ij}\|f\|_\B.
\end{equation}
By similar arguments as for \eqref{fRijint},  we have
\begin{equation}\label{fRijL2}
\int_{\wh X} R_{ij}\osc\bigl(f, \ B_{s_{ij}\e}(\cdot)\bigr)d\hn
\le \int_{\wh X}
\osc\bigl(f, \ B_{s_{ij}\e}(\cdot)\bigr)\mathbbold{1}_{U_{ij}}d\hn
\le s_{ij}^\a\e^\a |f|_\B.
\end{equation}
We note  that for each $j$, $\wh TU_{ij}=Q_0$ and the ``thickness'' of
$\wh T_{ij}^{-1}B_\e(\partial Q_0)$ is of order $s_{ij}\e$,
since $\partial Q_0$ consists of piecewise smooth surfaces.
So $G_{ij}(\e, \e_0)\le C_G \e s_{ij}$ for some $C_G$
independent of $i$ and $j$.
Therefore we have
\begin{equation*}%\label{fsum2}
\begin{split}
& \int_{\wh X} \e^{-\a} 2 (1+\dd \e^\a)G_{ij}(\cdot, \e, \e_0)
 \bigl[|f| + \osc(f, B_{\e_0}(\cdot))\bigr] d\hn  \\
\le &2(1+\dd \e^\a)C_G\e^{1-\a}s_{ij}\bigl[\|f\|_{L^1(\hn)} + \e_0^{\a}|f|_\B\bigr].
\end{split}
\end{equation*}
Hence by \eqref{fRija} we get that
\begin{equation}\label{fRijL3}
|R_{ij}f|_\B\le C_R's_{ij}^\a\bigl[\|f\|_{L^1(\hn)} +|f|_\B\bigr]
=C_R's_{ij}^\a\|f\|_\B
\end{equation}
for $C_R'=(1+\dd \e_0^\a)(1+2C_G\e_0^{1-\a})+2\dd C_b\hn(Q_0)$.

We now consider the case $n=i+1> N(\e)$.
As we mentioned in Remark~\ref{RAssmpe4}, in this case $m\ge 2$.
By definition, there is $C_s>0$ such that $\wh g(x_{ij})\le C_s^2 s_{ij}^2$
for any $x_{ij}\in \wh T_{ij}^{-1}B_{\e}(Q_0)$ with $j=2,\cdots, K$.
By Assumption T${}''$(e) we know that for any $x\in B_\e(Q_0)$,
$$
\Bigl(\sup_{\wh T_{i,j}^{-1}B_{\e}(x)} {\wh g}\Bigr)^{1/2}
\le  \Bigl(\sum_{\ell=N(\e)}^\infty
     \sup_{\wh T_{\ell j}^{-1}B_{\e}(x)} {\wh g}\Bigr)^{1/2}
\le \sqrt{b}\e^{(m+\a)/2}
\le \sqrt{b}\e^{\a}.
$$
Therefore we obtain
$$
\sup_{\wh T_{ij}^{-1}B_{\e}(x)} {\wh g}
=(\sup_{\wh T_{ij}^{-1}B_{\e}(x)} {\wh g}\Bigr)^{1/2}
 (\sup_{\wh T_{i,j}^{-1}B_{\e}(x)} {\wh g}\Bigr)^{1/2}
\le C_s s_{ij} \sqrt{b}\e^{\a}$$
and substitute in  \eqref{fRijb} to get ($\a\le 1$):
$$
|R_{ij}f|_\B\le C_R''s_{ij}\|f\|_\B\le C_R''s_{ij}^\a\|f\|_\B
$$
for $C_R''=2(\c_m\e_0^m)^{-1}\sqrt{b}C_s$.

Finally, by \eqref{fRijL1}, we have
\begin{eqnarray*}
\|R_{ij}f\|_1
\le \int R_{ij}|f|d\hn
\le C_b\nu(Q_0) s_{ij} \|f\|_\B.
\end{eqnarray*}
Thus we have
$\|R_{ij}f\|_\B=(C_R'+C_R''+C_b\nu(Q_0))s_{ij}^\a\|f\|_\B,$ which implies
the result of the lemma.
\end{proof}
We are finally ready to give the proof of Theorem D.

\begin{proof}[Proof of Theorem D]
We first choose $\e_0>0$ as in Lemma~\ref{LLYspV}, and
define $\B$ correspondingly by using that $\e_0$.
By Proposition~3.3 in \cite{Ss}, $\B$ is complete
and hence is a Banach space.
Then Assumption~B(a) to (f) follow from Lemma~\ref{LBsp2}.

By Lemma~\ref{LLYspV} we know that conditions $(S1)$ is satisfied.
Assumption~T${}''$(a), (d) and (c) imply Assumption~T~(a), (c) and (d)
respectively.  Assumption~T(b) is implied by the construction of the
first return map.
Lemma~\ref{LLYspV}(iii) gives \eqref{LYineq3}.
Therefore all conditions for Theorem~B are satisfied;
hence we obtain conditions $(S2)$ and $(S3)$.
The fact that $\|R_n\|= O(s_n^{\a})$ follows from Lemma~\ref{LR_ns}.
\end{proof}

\section{Multidimensional spaces: the role of the determinant in getting an optimal bound}
\label{Sapplicationb}
\setcounter{equation}{0}
%%%%%%%%%%%%%%%%%%%%%%%%%%%%%%%%%%%%%%%%%%%%%%%%%%%%%%%%%%%%%%%
%%%%%%%%%%%%%%%%%%%%%%%%%%%%%%%%%%%%%%%%%%%%%%%%%%%%%%%%%%%%%%%

In this section we put additional conditions on  the map $T$ that
we studied in the previous chapter in order to
get optimal estimates for the decay of correlations for observable supported in $\tilde{X}.$
As we anticipated in the Introduction, if $\|R_n\|$ decreases, in some norm,
 as $|\det DT^{-n}|$,
then it usually has the same order
as $\mu(\tau=n)$, which  approaches to $0$ faster than $\mu(\tau>n).$
Since $\sum_{k\ge n}\mu(\tau>k)$ gives the optimal decay rates of correlations
and $\sum_{k\ge n}\|R_k\|$ determines
the order of the error terms $F_{\b}(n)$,
we can get  lower estimates for decay rates.

\subsection{Assumptions and statement of the results.}%\label{SSmapft}
Let us suppose that $T$ satisfies Assumption T${}''$(a), (d) and (e) in the last section.
We replace part (b) and (c) by the following

\begin{AssumptionT''}
\begin{enumerate}
\item[{\rm (b${}'$)}] {\rm (Fixed point and a neighborhood)}
There is a fixed point $p\in U_1$ and a neighborhood $V$ of $p$ such that
$T^{-n}(V)\cap \partial U_j=\emptyset$ for any $j=1, \dots,K$
and for any $n\ge 0$.
%$T^{-1}\{p\}\notin \partial U_j$ for any $j$.

\item[{\rm (c${}'$)}] {\rm (Topological exactness)}
$T: X\to X$ is topologically exact, that is, for any $x\in X$, $\e>0$,
there is an $\wt N=\wt N(x,\e)>0$ such that $T^{\wt N}B_\e(x)=X$.
\end{enumerate}
\end{AssumptionT''}

%For convenience we also regard AssumptionT${}''$(a${}'$), (d${}'$) and
%(e${}'$) are the same as AssumptionT${}''$(a), (d) and (e) respectively.

\begin{Remark}\label{Rexist}
Clearly maps with a Markov partition, even countable,  satisfy Assumption T''(b') provided the
neutral fixed point is in the interior of a partition element.
In Exercise 5.5 we will introduce a class of non-Markov maps
that satisfy T''(b') as well.

\end{Remark}

\begin{Remark}
Assumption T${}''$(b${}'$) will allow us to get a better estimate for
$\|R_n\|_\B$ which in turn will give us  optimal bounds. To understand the difference with the results of Section 4, we recall that there, starting from
\eqref{fRijL2}, we got   the estimate in \eqref{fRijL3}
$|R_{ij}f|_\B \le C_R's_{ij}^\a\|f\|_\B$ for some
constant $C_R'>0$, and hence
$\|R_{ij}f\|_\B$ decreases as the speed of $s_{ij}^\a$ does. This was precisely the statement
of Lemma~\ref{LR_ns}, where
$s_{ij}$ was given by the norm $\|D{\wh{T}}_{ij}^{-1}\|$ of the derivatives.
With Assumption T${}''$(b${}'$) and by considering a different and smaller Banach space  we can get the new estimates \eqref{fRijosc},
which lead to
the upper bound $|R_{ij}f|_\Q \le C_2' d_{ij}\|f\|_\B$
in \eqref{fRijQ},
where $d_{ij}$ is given by the determinant $|\det D{\wh{ T}}_{ij}^{-1}|$.
On the other hand, estimates of the norm $|R_{ij}f|_\H$
can be obtained and decrease with the same order.
 Other explications and details  will be given in the proof.
\end{Remark}

%\begin{Remark}\label{Rexact}
%Clearly, topological exactness implies topological mixing.
%\end{Remark}

Since we want to reserve the symbol $\mathcal{B}$ for the functional space upon which we want to get the renewal type results leading to the bounds on the decay of correlations, we begin to rename   the seminorm and the Banach space defined in \eqref{fdefB} and  \eqref{fnorm''} with
  $\Q$, instead of $\B$.  We remind that such a seminorm  will   depend on $\alpha$ and on $\epsilon_0$, the latter dependence affecting only the value of the seminorms.
Then  \eqref{fnorm''} will be now written as:
$$
\|f\|_{\Q}=\|f\|_{L^1(\hn)}+|f|_{\Q}.
$$

Recall that $V$ is a neighborhood of $p$ given in Assumption T$\,{}''$(b').
We denote the preimages $T_{i_k}^{-1}\dots T_{i_1}^{-1} V$ by
$V_{i_1\dots i_k}$ or $V_I$ where $I=i_1\dots i_k$.
We also denote with $\I$ the set of all possible words $i_1\cdots i_k$
such that $T_{i_k}^{-1}\dots T_{i_1}^{-1} V$ is well defined,
where $i_k\in \{1,\cdots, K\}$ and $k>0$.

For an open set $O$, let $\H:=\H^\a_{\e_1}=\H^\a_{\e_1}(O,H)$ be
the set of H\"older functions $f$ on $O$ that satisfy
$|f(x)-f(y)|\le Hd(x,y)^\a$ for any $x,y\in O$ with $d(x,y)\le \e_1$.

Let $\h$ be a fixed point of the transfer operator $\wh \P$,
which will be unique under the assumptions of the theorem below.
We now define $\B$ by
\begin{equation}\label{fdefspBH}
\B:=\B_{\e_0,\e_1}^{\alpha}=\left\{ f\in \Q: \exists H>0 {\text\ s.t.}\
(f/\h)|_{V_I}\in \H^\a_{\e_1}(V_I,H) \  \forall I\in \I\right\},
\end{equation}
and for any $f\in \B$, let
\begin{equation*}%\label{fdefnormBH}
|f|_{\H}:=|f|_{\H^\a_{\e_1}}=\inf\{H: (f/\h)|_{V_{I}}\in \H^\a_{\e_1}(V_I,H)
\ \forall I\in \I\}.
\end{equation*}
Sublemma~\ref{SLhbound} and \ref{SLhHolder} below imply
that $\h>0$ on all $V_{I}$, and therefore the definition makes sense.
Then we take $|\cdot|_{\Q}+|\cdot|_{\H}$ as a seminorm
for $f\in \B$ and define the norm in $\B$ by
\begin{equation}\label{fnorm'''}
\parallel\cdot\parallel_{\B}
= \|\cdot \|_1+|\cdot|_{\Q}+|\cdot|_{\H}.
\end{equation}
Clearly, $\B\subset \Q$ and
$\|f\|_\B\ge \|f\|_\Q$ if $f\in \B$.

We now remind that for any sequences of numbers $\{a_n\}$ and $\{b_n\}$,
we use
$a_n\sim b_n$ if $\disp\lim_{n\to \infty}a_n/b_n=1$, and
$a_n\approx b_n$
if $c_1b_n\le a_n\le c_2b_n$ for some constants $c_2\ge c_1>0$.

Let $d_{ij}=\sup\bigl\{|\det D\wh T^{-1}_{ij}(x)|: x\in B_{\e_0}(Q_0)\bigr\}$,
and $d_n=\max\bigl\{d_{n-1,j}: j=2,\cdots, K\bigr\}$.

\begin{TheoremE}\label{ThmE}
Let $\wh X$, $\wh T$ and $\B$ be defined as above and
suppose that  $T$ satisfies Assumption~T$\,{}''$(a), (b$\,{}'$), (c$\,{}'$),
(d) and (e).
Then there exist $\e_0\ge \e_1>0$ such that Assumption B(a) to (f)
and conditions $S(1)$ to $S(4)$  are satisfied and
$\|R_n\|_\B = O(d_n^{m/(m+\a)})$.
Hence, if $\sum_{k=n+1}^\infty d_n^{m/(m+\a)}= O(n^{-\b})$ for some $\b>1$,
then there exists $C>0$ such that for any functions
$f\in \B$, $g\in L^\infty(X,\nu)$ with
$\supp{f}, \; \supp{g}\subset {\wh X}$,
(\ref{fThmA}) holds.

Moreover, if $T$ satisfies \eqref{B.0} near $p=0$,
then $\disp\sum_{k=n+1}^\infty \mu(\tau>k)\approx n^{-(\frac{m}{\c} -1)}$.
In this case, if $d_n=O(n^{-\b'})$ for some $\b'>1$ and if
\begin{eqnarray}\label{fThmE1}
\b=\b'\cdot \frac{m}{m+\a}-1> \max\{2,\frac{m}{\c}-1\},
\end{eqnarray}
then
\begin{eqnarray}\label{fThmE2}
\Cov(f, g\circ T^n)
\sim \sum_{k=n+1}^\infty \mu(\tau>k)\int fd\mu\int gd\mu\
\approx \frac{1}{n^{\frac{m}{\c} -1}}.
\end{eqnarray}
In particular, if Assumption T$\,{}''$(e$\,{}'$)
in Section 4.1 stating bounded distortion
also holds,
then the above statements remain true if we replace $m/(m+\a)$
in \eqref{fThmE1} by $1$.
\end{TheoremE}

\begin{Remark}%\label{Rdist}
Whenever $T$ satisfies \eqref{B.0} near $p$,
Assumption~T$\,{}''$(c$\,{}'$) implies that $h$ is bounded away from $0$
on the sets $\{\tau>n\};$  hence $\mu(\tau>n)$ and $\nu(\tau>n)$
have the same order
and $\sum_{k=n+1}^\infty \mu(\tau>k)\approx n^{-(\frac{m}{\c} -1)}$.
This is the case in Example 5.1, 5.2 and 5.4 below.

On the other hand, if Assumption~T$\,{}''$(c$\,{}'$) only holds
for an invariant subset of $X$ like in Example 5.3,
then $\h$ may be only supported on a part of the set
$\{\tau>n\}$, and therefore $\mu(\tau>n)$ may decrease faster.
In this case, $\sum_{k=n+1}^\infty \mu(\tau>k)=o(n^{-(\frac{m}{\c} -1)})$.
\end{Remark}

\subsection{Examples}\label{SSmapft}
Before giving the proof, we present  a few examples.
The first four examples concern various decay rates,
where we will always assume that $T$ satisfies
Assumption~T$\,{}''$(a), (b$\,{}'$), (c$\,{}'$) and (d).
Example~\ref{example5} and thereinafter are for maps
satisfying Assumption~T$\,{}''$(b$\,{}'$).

\begin{Example}\label{example1}
Let us assume $m=3$, and near the fixed point $p=(0,0,0)$,
the map $T$ has the form
\begin{equation*}%\label{f2.1}
T(w)=  \bigl( x(1+|w|^2+O(|w|^3)), \ y(1+|w|^2+O(|w|^3)),
                  z(1+2|w|^2+O(|w|^3) \bigr)
\end{equation*}
where $w=(x,y,z)$ and $|w|=\sqrt{x^2+y^2+z^2}$.
\end{Example}

This map is very similar to that studied in Example 1 in \cite{HV},
although it is now in a three dimensional space.
We can still use the same arguments to show that
Assumption T${}''$ (e) is satisfied.

We set $w_n=T_1^{-n}w$;
clearly, $|w|+|w|^3+O(|w|^4)\le |T(w)|\le |w|+2|w|^3+O(|w|^4)$.
By standard arguments we know that
\begin{eqnarray*}%\label{f2.4}
\frac{1}{\sqrt{4n}}+O\Bigl(\frac{1}{\sqrt{n^3}}\Bigr)
\le |w_n |
\le \frac{1}{\sqrt{2n}}+O\Bigl(\frac{1}{\sqrt{n^3}}\Bigr)
\end{eqnarray*}
(see also Lemma 3.1 in \cite{HV}).
Since we are in  a three dimensional space, we now have
$\disp \nu(\tau>k)\approx \frac{1}{k^{m/\c}}=\frac{1}{k^{3/2}}$, and therefore
$\disp \sum_{k=n+1}^\infty\nu(\tau>k)\approx \frac{1}{n^{1/2}}$.

%$DT(w)$ has the form
%\begin{eqnarray*}%\label{f2.2}
%\left( \begin{array}{ccc}
%   \!\!\! 1+ 3x^2+ y^2 +z^2 \!\!\!&  2xy &  2xz   \\
%   2xy & \!\!\! 1+ x^2+ 3y^2 +z^2  \!\!\!&  2yz \\
%   4xz   & 4yz &\!\!\! 1+ 2x^2+ 2y^2 + 6z^2 \!\!\!
%   \end{array} \right)+O(|w|^3)
%\end{eqnarray*}
%and hence
%\begin{eqnarray*}%\label{2.3}
%\det DT(w)=1+6x^2+6y^2+8z^2+O(|w|^3).
%\end{eqnarray*}
It is easy to see that $\det DT(w)=1+6x^2+6y^2+8z^2+O(|w|^3)$.
So we have $|\det DT_1^{-1}(w)|\le 1-6|w|^2+O(|w|^3)$.
By Lemma 3.2 in \cite{HV} with $r(t)=1-6t^2+O(t^3)$, $\c=2$, $C'=6$ and $C=1$,
we  get that $|\det DT_1^{-n}(x)|=O(1/n^3)$.
Hence we have $\b'=3$ and $\b=3m/(m+\a)-1> 5/4$.
Since $m/\c-1=1/2$, \eqref{fThmE1} holds, and therefore
we have \eqref{fThmE2} with the decay rate of order $1/n^{\frac12}$; contrarily to Example 4.1, we now got an optimal bound.

\begin{Example}\label{example2}
Assume $m=2$, and near the fixed point $p=(0,0)$,
the map $T$ has the form
\begin{equation*}%\label{f2.1}
T(z)= \bigl( x(1+|z|^\c+O(|z|^{\c'})), \ y(1+2|z|^\c+O(|z|^{\c'})) \bigr)
\end{equation*}
where $z=(x,y)$, $|z|=\sqrt{x^2+y^2}$, $\c\in (0,1)$ and $\c'>\c$.
\end{Example}

By methods similar to Example 1 in \cite{HV} we can check that
Assumption T${}''$ (e) is satisfied.
Denote $z_n=T_1^{-n}z$.
Since $|z|+|z|^{1+\c}+O(|z|^{\c'})\le |T(z)|\le |z|+2|z|^{\c+1}+O(|z|^{\c'})$,
we have
\begin{eqnarray*}%\label{f2.4}
\frac{1}{(2\c n)^{1/\c}}+O\Bigl(\frac{1}{n^\d}\Bigr)
\le |z_n |
\le \frac{1}{(\c n)^{1/\c}}+O\Bigl(\frac{1}{n^\d}\Bigr)
\end{eqnarray*}
for some $\d>1/\c$.
So $\disp \nu(\tau>k)\approx \frac{1}{k^{2/\c}}$, and therefore
$\disp \sum_{k=n+1}^\infty\nu(\tau>k)\approx \frac{1}{n^{\frac{2}{\c}-1}}$.

%%\begin{eqnarray*}%\label{f2.2}
%\disp DT(z)=\left( \begin{array}{cc}
%   1+ \disp \frac{(1+\c)x^2+y^2}{|z|^{2-\c}}& \disp\frac{\c xy}{|z|^{2-\c}}\\
% \disp\frac{2\c xy }{|z|^{2-\c}} & 1+ \disp\frac{2x^2+2(1+\c)y^2}{|z|^{2-\c}}
%   \end{array} \right)+O(|z|^{\c'}),
%\end{eqnarray*}
%and hence
%\begin{eqnarray*}%\label{2.3}
%\det DT(z)=1+\frac{(3+\c)x^2+(3+2\c)y^2}{|z|^{2-\c}}+O(|z|^{\c'}).
%\end{eqnarray*}
It is possible to show that
$\disp |\det DT(z)|=1+\frac{(3+\c)x^2+(3+2\c)y^2}{|z|^{2-\c}}+O(|z|^{\c'})$.
Therefore  $|\det DT_1^{-1}(z)|\le 1-(3+\c)|z|^\c+O(|z|^{\c'})$,
and $|\det DT_1^{-n}(z)|=O(1/n^{1+3/\c})$.
Hence $\b'=1+\c/3$ and
$\b=(1+3/\c)\cdot 2/(2+\a)-1>2/\c-1$.
Therefore \eqref{fThmE1} holds, and the decay rates is of order $1/n^{\frac{2}{\c} -1}$.

\begin{Example}\label{example3}
Assume $m=2$, and take the same map as in Example 1 in \cite{HV},
namely, near the fixed point $p=(0,0)$, the map $T$ has the form
\begin{eqnarray*}%\label{f2.1}
T(x,y)=  \bigl( x(1+x^2+y^2), \ y(1+x^2+y^2)^2 \bigr).
\end{eqnarray*}

The map allows an infinite absolutely continuous invariant measure.
However, it can be arranged in such a way that there
is an invariant component that supports a finite
absolutely continuous invariant measure $\mu$.
Near the fixed point, the region supporting this component has the form
$$
\{z=(x,y): |y|< x^2\}.
$$
We may regard $X$ as this component, and $T: X\to X$ satisfies
the assumptions.
\end{Example}

We can check that the map has bounded distortion near the fixed point
restricted to this region.
Hence, the map verifies Assumption T${}''$(e${}'$).

Since $|z_n|=O(1/\sqrt{n})$ and for $z=(x,y)$, $|y|\le x^2$,
we  get $\disp \nu(\tau>k)\approx \frac{1}{k^{3/2}}$, and
$\disp \sum_{k=n+1}^\infty\nu(\tau>k)\approx \frac{1}{n^{1/2}}$.

On the other hand, $|\det DT(z)|=1+5x^2+7y^2+O(|z|^4)$.
Since $|y|\le x^2$, $|z|=|x|+O(|z|^2)$; thus
 $|\det DT(z)|=1+5|z|^2+O(|z|^4)$, and therefore
$|\det DT_1^{-n}(z)|=O(1/n^{5/2})$.  So $\b'=5/2$ and $\b=3/2$.
We obtain that the decay rate is of order $1/n^{1/2}$.

\begin{Example}\label{example4}
Assume $m\ge 3$ and near the fixed point $p=(0,0,0)$,
the map $T$ has the form
\begin{equation*}%\label{f2.1}
T(z)=  z\bigl( 1+|z|^\c+O(|z|^{\c+1})\bigr),
\end{equation*}
where $m> \c>0$.
\end{Example}

These examples are comparable with those  in Example~4.1, except for the stronger topological assumptions which we now put on the maps.
We know that those maps satisfy Assumption, T${ }''$(e${}'$).

We set $z_n=T_1^{-n}z,$ then
we have $|z_n|=1/(n\c)^{1/\c}+O\bigl(1/(n\c)^{\frac{1}{\c}+1}\bigr)$ and
$|\det DT(z)|=1+(m+\c)|z|^\c +O\bigl(|z|^{\c+1}\bigr)$.
Hence, we get that $|\det DT_1^{-n}|\approx 1/n^{\frac{m}{\c}+1}$,
(for the relative computations see  Lemma 3.1 and 3.2 in \cite{HV}).
Therefore  $\b'=\frac{m}{\c}+1$ and $\b=m/\c$.

On the other hand, we see that
$\nu(\tau>k)=O\bigl(1/k^{m/\c}\bigr)$, and then
$\disp \sum_{k=n+1}^\infty\nu(\tau>k)\approx \frac{1}{n^{\frac{m}{\c}-1}}$.
Since $m>\c$, the invariant measure $\mu$ is finite and $\b>1$.
We get that the decay rate is
of order $1/n^{\frac{m}{\c }-1}$.

\begin{Example}\label{example5}
Let us take $X=[-100, 100]$ and  a partition
$\xi= \{U_0, U_i^+, U_i^-:  i=1, \dots ,9\}$ of $X$ into 19 subintervals
such that $U_0=[-10,10]$, $U_i^-=[-10i-10, -10i)$ and $U_i^+=(10i, 10i+10]$.
Also set $\disp \partial\xi=\cup_{U\in \xi}\partial U$.

We then define a piecewise smooth expanding map $T: X\to X$ with an indifferent
fixed point $p=0$ as following:

\begin{enumerate}
\item[{\rm (i)}]
$T(\intset U_i^\pm)=\intset X$ for $i\not= -8, 8$ and $|T_i'(x)|\ge 10$
for all $x\notin [-3, 3]\cup \partial\xi$;

\item[{\rm (ii)}]
$T(x)=x+4|x|^{1.5}$ for $x\in [-3, 3]$;

\item[{\rm (iii)}]
$T$ is increasing on $U_{9}^\pm$ and maps $\intset U_{9}^\pm$
to $\intset X$ linearly, that is, $T(x)=20(x-95)$ on $U_9^+$ and
$T(x)=20(x+95)$ on $U_{9}^-$;

\item[{\rm (iv)}] $T(U_8^-)=[-100, e_+)$ and $T(U_{8}^+)=(e_-, 100]$,
where $e_\pm\in E_\pm$, and
$E_\pm=\{x\in U_{9}^\pm: T^n(x)\in U_{9}^+\cup U_9^-  \  \forall n\ge 0\}$.
\end{enumerate}
\end{Example}

It is clear that $T$ satisfies Assumption~T${ }''$(a), (b), (c${}'$), (d)
and (e${}'$); moreover (iv) above shows that the partition $\xi$ is not Markov.
By the choice of $E_\pm$, the orbits $\{T^n(e_\pm): n > 0\}$
are contained in $E_+\cap E_-$, and therefore in $U_9^+\cup U_9^-$.
Note that all possible image sets
$\{T^n(U): U\in \vee_{i=0}^{n-1} T^{-i}(\xi)\}$
have the form $[-100, 100]$, $[-100, T^n(e_\pm)]$, $[T^n(e_\pm), 100]$
or $[T^n(e_\pm), T^n(e_\mp)]$
up to the endpoints.
So if we take $V=[-2,2]$, then $V\cap T^k(\partial U)=\emptyset$
for any $U\in \xi$ and $k\ge 0$.  It follows that
$T^{-k}(V)\cap \partial U=\emptyset$ for any $U\in \xi$ and $k\ge 0$.
Hence, Assumption~T${ }''$(b${}'$) holds.

\begin{Remark}
We mention here that $T|_{U_9^\pm}$ do not have to be linear.
Also, the role of $U_8^\pm$ and $U_9^\pm$ can be replaced by any pairs
$U_i^\pm$ and $U_j^\pm$ for $i,j\not=0$ and $i\not=j$.
\end{Remark}

\medskip
The same idea can be used to generate example of maps
in higher dimensional spaces.  For example, in the plane we can take
$X=[-100, 100]\times [-100, 100]$, and partition $X$ in to squares
$U_{ij}^{\pm\pm}$ of size $10\times 10$, except for
$U_0=[-10, 10]\times[-10, 10]$.
Near the origin we can define
$T(x,y)=  \bigl( x(1+x^2+y^2), \ y(1+x^2+y^2)^2 \bigr)$ as in
Example~\ref{example3}.
Then we let $U_{i, 9}^{\pm\pm}$ and $U_{i, 8}^{\pm\pm}$,
or $U_{9,j}^{\pm\pm}$ and $U_{8,j}^{\pm\pm}$, or both,
where $i, j=\pm 0, \pm 1,\dots \pm 9$,
will play the same role as $U_9^\pm$ and $U_8^\pm$ in the above example.
That is, the map can be arranged in such a way that under $T^n$
the images of the boundaries of all sets in the partition
are contained in the region $\{(x,y)\in X: 90\le |y|\le 100\}$
or $\{(x,y)\in X: 90\le |x|\le 100\}$, or both.
By this way, we can construct a map $T$ that satisfies all conditions
given by Assumption~T${ }''$(a), (b${}'$), (c${}'$), (d) and (e).

\medskip
In fact, systems satisfying Assumption~T${ }''$(a), (b${}'$), (c${}'$),
and (d) are dense in the set of the systems satisfying
Assumption~T${ }''$(a), (b), (c${}'$) and (d) in the $C^1$ topology.
This means that for any system satisfying Assumption~T${ }''$(a), (b), (c${}'$)
and (d), we can make an arbitrarily small $C^1$ perturbation to get
a map $\overline{T}$ such that there exists a small neighborhood $V$ of $p$
with $\overline{T}^{-n}(V)\cap \partial U_j=\emptyset$
for any $j=1, \dots, K$ and for any $n\ge 0$.
To see this, we first note that for any fixed $n_0$, we can get
that $\overline{T}^{-n}(p)\cap \partial U_j=\emptyset$ for any $0<n\le n_0$
by using a small perturbation, and then get that
$\overline{T}^{-n}V\cap \partial U_j=\emptyset$ for any $0<n\le n_0$
by taking $V$ small enough.
Further, for any connected component $V^{(n)}_i$ of $\overline{T}^{-n}V$,
we require that $d(V^{(n)}_i, \partial U_j)\ge \diam V^{(n)}_i$
for any $j=1, \dots, K$.
Now we consider the case $n>n_0$.
If $V^{(n)}_i\cap \partial U_j\not=\emptyset$,
then we can use a small perturbation $\phi^{(n)}_i$ with both
$d(\phi^{(n)}_i, \id)$ and $\|D\phi^{(n)}_i\|$ small enough
to get $d(V^{(n)}_i, \partial U_j)\ge \diam V^{(n)}_i$.
Notice that Assumption~T${ }''$(d) implies $s<\1/4$.
It is easy to see that if $V^{(n_2)}_{i_2}$
intersects the ($2\diam V^{(n_{1})}_{i_1}$)-neighborhood of some $V^{(n_{1})}_{i_1}$
with $n_2>n_1$, then $\diam V^{(n_2)}_{i_2}<(1/4)\diam V^{(n_1)}_{i_1}$.
Hence, we can require $d(\phi^{(n)}_i, \id)$ and $\|D\phi^{(n)}_i\|$
decrease with $n$ at least by a fact $1/4$ at each step.
Then after a sequence of perturbations
we still have  $d(V^{(n)}_i, \partial U_j)\ge (1/2)\diam V^{(n)}_i$
for any $n>0$ and the $C^1$ norm of the composition of the sequence of
perturbations are still small.
Hence the resulting map $\overline{T}$ satisfies Assumption~T${ }''$(b${}'$),
and obviously satisfies Assumption~T${ }''$(a), (c${}'$), and (d) as well.
We leave the details to the reader.

%%%%%%%%%%%%%%%%%%%%%%%%%%%%%%%%%%%%%%%%%%%%%%%%%%%%%%%%%%%%%%%%
%%%%%%%%%%%%%%%%%%%%%%%%%%%%%%%%%%%%%%%%%%%%%%%%%%%%%%%%%%%%%%%%
\subsection{Proof of  Theorem E}%\label{SSmapft}
%%%%%%%%%%%%%%%%%%%%%%%%%%%%%%%%%%%%%%%%%%%%%%%%%%%%%%%%%%%%%%%%
%%%%%%%%%%%%%%%%%%%%%%%%%%%%%%%%%%%%%%%%%%%%%%%%%%%%%%%%%%%%%%%%

\begin{proof}[Proof of Theorem E]
We begin to choose  $\e_0>0$ satisfying Lemma~\ref{LLYspV} in the previous section,
and then we take $\e_1\in (0,\e_0]$ as in Lemma~\ref{LLYspH} below.
We reduce $\e_1$ further if necessary such that
$\eta':=\eta+D_\H(\e_0)\e_1^\a<1$, where $\eta<1$ is given in
Lemma~\ref{LLYspV} and $D_\H(\e_0)>0$ is given in Lemma~\ref{LLYspH}.
Then we take $\B:=\B_{\e_0, \e_1}^{\alpha}$ as in \eqref{fdefspBH};
with the norm given in \eqref{fnorm'''},
$\B$ satisfies Assumption~B(a) to (f) by Lemma~\ref{LBsp2BH}.

Thanks to  Lemmata~\ref{LLYspV} and \ref{LLYspH}, condition S(1)
 is satisfied with constants $\eta$ and $D$ replaced by
$\eta',$ defined as above, and $D+D_\H(\e_0)\e_1^\a$ respectively,
where $D$ is the number given in Lemma~\ref{LLYspV}.

Assumption~T${}''$(a), (d) and (c${}'$) imply Assumption~T~(a), (c) and (d)
respectively.  Assumption~T(b) follows from  the construction of the
first return map.
Lemma~\ref{LLYspV}(iii) and \ref{LLYspH}(iii) give \eqref{LYineq3}.
Therefore all the conditions for Theorem~B are satisfied; hence we obtain conditions $S(3)$ and $S(4)$.

The facts that $\|R_n\|_\B= O(d_n^{m/(m+\a)})$,
and $\|R_n\|_\B= O(d_n)$ if Assumption~T${}''$(e${}'$)
is satisfied, follow from Lemma~\ref{LR_n}:
therefore we have established the decay of correlations \eqref{fThmA}.

If $T$ also satisfies \eqref{B.0}, then we know that
for any $z$ close to $p$, $|T_1^{-n}z|$ is of order $n^{-1/\c}$.
Hence $\hat\nu\{\tau >k\}$ has the order $k^{-m/\c}$,
and $\sum_{k=n+1}^\infty k^{-\frac{m}{\c}}=O(n^{-\frac{m}{\c}+1})$.
Then the rest of the theorem is clear.
\end{proof}

\begin{Lemma}\label{LBsp2BH}
$\B$ is a Banach space satisfying Assumption B(a) to (f) with
$C_a=2C_b=2\c_m^{-1}\e_0^{-m+\a}$, where $\c_m$ is the volume of the
unit ball in $\mathbb R^m$.
\end{Lemma}

\begin{proof}

We already know  that $\Q$ is a Banach space, and the proof of the completeness of  $\B$ follows from standard arguments.
%\\{\bf WARNING: I think this is the way to say(prove) it; your previous argument that $\B$ was closed in $\Q$ was not enough. The proof is more or less the same to prove completeness for Holder + $L^1$.}\\

Now we verify Assumption B(a) to (f).

By Lemma~\ref{LBsp2}, the unit ball of $\Q$ is compact in $L^1(\wh X, \hn)$. Since $||f||_{\B}\ge ||f||_{\Q}$ for any $f\in \B\subset \Q$, the unit ball of $\B$ is contained in the unit ball of $\Q$. Since $\B$ is closed in $\Q$, the unit ball of $\B$ is also compact.
This is Assumption B(a).

Moreover, for any $f\in \Q$,
$\|f\|_\infty\le C_b\|f\|_\Q\le C_b\|f\|_\B$ with $C_b=\c_m^{-1}\e_0^{-m+\a}$; we have thus got  Assumption B(b).

By invoking again  Lemma~\ref{LBsp2}, we have,  for any $f, g\in \Q:$ $\|fg\|_\Q\le C_a\|f\|_\Q\|g\|_\Q$,
where $C_a=2\c_m^{-1}\e_0^{-m+\a}=2C_b$.
It is easy to check that
$$|fg|_\H\le \|f\|_\infty |g|_\H+\|g\|_\infty |f|_\H
\le C_b\|f\|_\Q |g|_\H+C_b\|g\|_\Q |f|_\H.
$$
Hence,
\begin{equation*}%\label{fR^1}
\begin{split}
&\|fg\|_\B= \|fg\|_\Q +|fg|_\H
\le C_a\|f\|_\Q\|g\|_\Q+ C_b\|f\|_\Q |g|_\H+C_b\|g\|_\Q |f|_\H  \\
\le&C_a\bigl(\|f\|_\Q+|f|_\H\bigr)\bigl(\|g\|_\Q+|g|_\H\bigr)
= C_a \|f\|_\B \|g\|_\B.
\end{split}
\end{equation*}
Therefore  Assumption B(c) follows with $C_a=2\c_m^{-1}\e_0^{-m+\a}=2C_b$.

Similarly, part (d) of Assumption B follows from the fact that
$\B$ contains all H\"older functions, which are in turn dense in $L^1(\wh X, \hn)$.

Assume $f(x)=\lim_{n\to\infty} f_n(x)$ for $\hn$-a.e. $x\in \wh X$.
By the proof of Lemma~\ref{LBsp2} we have
$|f|_\Q\le \liminf_{n\to\infty}|f_n|_\Q;$ moreover
for any $y,z\in V_I$, where $I\in \I$,
$$
\frac{|f(y)-f(z)|}{d(y,z)^\a}
\le \lim_{n\to\infty}\frac{|f_n(y)-f_n(z)|}{d(y,z)^\a}
\le \liminf_{n\to\infty}|f_n|_\H.
$$
Therefore $|f|_\H\le \liminf_{n\to\infty}|f_n|_\H;$ since $|f|_\B=|f|_\Q+|f|_\H$, we get part (e).

Since $\B\subset \Q$, part (f) follows directly from the fact that
$\Q$ satisfies Assumption~B(f).
\end{proof}

\begin{Lemma}\label{LLYspH}
Let $\e_0$ be as in Lemma~\ref{LLYspV}.
There exists $D_\H=D_\H(\e_0), \bar D_\H=\bar D_\H(\e_0)>0$ and
$\e_-\in (0, \e_0]$ such that for any $\e_1\in (0,\e_-]$, and by using the notation for the Banach space introduced in (\ref{fdefspBH}):
\begin{enumerate}
\item[{\rm (i)}] for any $f\in \B_{\e_0,\e_1}^\a$,
$%\begin{eqnarray}\label{fLYineq}
|{\wh \P} f|_{\H_{\e_1}}
\le s^\a |f|_{\H_{\e_1}}+D_\H \e_1^\a\|f\|_{\Q_{\e_0}};
$%\end{eqnarray}

\item[{\rm (ii)}] for any $f\in \B_{\e_0,\e_1}^\a$,
$%\begin{eqnarray}\label{LYineq2}
|R(z)f|_{\H_{\e_1}}\le |z|\bigl(s^a|f|_{\H_{\e_1}}
+\bar D_\H \e_1^\a\|f\|_{\Q_{\e_0}}\bigr);
$%\end{eqnarray}

\item[{\rm (iii)}] and for any $f\in \wt\B_{\e_0,\e_1}^\a$
$%\begin{eqnarray}\label{LYineq2}
|{\wt \P} \wt f|_{\wt \H_{\e_1}}
\le s^\a|\wt f|_{\wt \H_{\e_1}} +D_\H \e_1^\a\|\wt f\|_{\wt\Q_{\e_0}}.
$%\end{eqnarray}
\end{enumerate}
\end{Lemma}

\begin{proof}
(i) Let $\e_*\in (0,\e_0]$, $J_{\h}>0$ as in the proof of
Sublemma~\ref{SLhHolder} below.
Suppose $\e\in(0,\e_*]$, and $|f|_{\H_{\e_1}}= H$ for some $f$.
Take $x,y\in V_I$ for some $I\in \I$ with $d(x,y)=\e\le \e_*$.
Then by Assumption T${}''$(e), we can take $\dd>0$, $N=N(\e)>0$ for $b=1$.
Notice that
\begin{equation}\label{fLYspaceH1}
\begin{split}
\frac{{\wh \P}f(x)}{\h(x)}-\frac{{\wh \P}f(y)}{\h(y)}
=&\sum_{j=1}^{K}\sum_{i=1}^{\infty} \frac{\hat g(x_{ij})\h(x_{ij})}{\h(x)}
\Bigl(\frac{f(x_{ij})}{\h(x_{ij})}-\frac{f(y_{ij})}{\h(y_{ij})}\Bigr) \\
+&\sum_{j=1}^{K}\sum_{i=1}^{N} \frac{f(y_{ij})}{\h(y_{ij})}
\Bigl(\frac{\hat g(x_{ij})\h(x_{ij})}{\h(x)}
    - \frac{\hat g(y_{ij})\h(y_{ij})}{\h(y)} \Bigr) \\
+&\sum_{j=1}^{K}\sum_{i=N+1}^{\infty} \frac{f(y_{ij})}{\h(y_{ij})}
\Bigl(\frac{\hat g(x_{ij})\h(x_{ij})}{\h(x)}
    - \frac{\hat g(y_{ij})\h(y_{ij})}{\h(y)} \Bigr). \\
\end{split}
\end{equation}

Since $|f|_\H=H$, we have  $|{f(x_{ij})}/{\h(x_{ij})}-{f(y_{ij})}/{\h(y_{ij})}|
\le Hd(x_{ij},y_{ij})^\a \le s^\a Hd(x,y)^\a$.
Now, ${\wh \P}\h=\h$ implies
\begin{equation}\label{fNormal}
\sum_{j=1}^{K}\sum_{i=1}^{\infty} {\hat g(x_{ij})\h(x_{ij})}/{\h(x)}=1.
\end{equation}
Thus the first sum in (\ref{fLYspaceH1}) is bounded by
$s^\a Hd(x,y)^\a\le s^\a|f|_\H d(x,y)^\a$.

Note that by our assumption, $V_{I}$ does not intersect discontinuities.\footnote{This implies that the potential $\hat{g}_{ij}$ of the transfer operator is continuous. Such a potential has in fact  the form $\hat{g}_{ij}(x)=|\det D\hat{T}_{ij}(x)|^{-1},$  where $\hat{T}_{ij}=T^i_1 T_j,$ being $T_1$ and $T_j$ different determinations of the map $T.$ In the computation of the transfer operator, $\hat{g}$ is computed in the point $T^{-1}_jT^{-i}_1 x$, where $x$ belongs to the  sets of H\"older continuity $V_I$ which are in turn  the preimages of $V$.
The continuty of the potential is necessary to get the invariance of the new Banach space under the action of $\wh \P$.}
By Sublemma~\ref{SLhHolder}, $\h(y)/\h(x)\le e^{J_{\h}d(x,y)^\a}$,
and by Assumption T${}''$(e),
$\hat g(y)/\hat g(x)\le e^{\dd d(x,y)^\a}$ if $i\le N(\e)$.
So
$[\hat g(y_{ij})\h(y_{ij})/\h(y)]/[\hat g(x_{ij})\h(x_{ij})/\h(x)]\le  e^{\dd'd(x,y)^\a}$
for some $\dd'>0$.
We take $\e_-\in (0,\e_*]$ small enough such that
$e^{\dd \e_1^\a}-1\le 2\dd'\e_1^\a$ for any $\e_1\le (0,\e_-]$.  Then for
$d(x,y)=\e\le \e_1$, we have
\begin{equation}\label{fRijL1norm0}
\Bigl|\frac{\hat g(x_{ij})\h(x_{ij})}{\h(x)}
    - \frac{\hat g(y_{ij})\h(y_{ij})}{\h(y)}\Bigr|
\le 2\dd'\frac{\hat g(x_{ij})\h(x_{ij})}{\h(x)}\cdot d(x,y)^\a.
\end{equation}
Therefore by \eqref{fNormal}, the second sum in \eqref{fLYspaceH1} is bounded by
\begin{equation*}%\label{fRijL1norm}
\sum_{j=1}^{K}\sum_{i=1}^{N} \frac{f(y_{ij})}{\h(y_{ij})}
\frac{\hat g(x_{ij})\h(x_{ij})}{\h(x)}\cdot 2\dd'd(x,y)^\a
\le 2\dd' \h_*^{-1}\|f\|_\infty d(x,y)^\a,
\end{equation*}
where $\h_*$ is the essential lower bound of $\h$ given
by Sublemma~\ref{SLhbound}.

By  Assumption T${}''$(e), the third sum in \eqref{fLYspaceH1}
is bounded by
\begin{equation*}%\label{fRijL1norm}
\begin{split}
&\sum_{j=1}^{K}\sum_{i=N+1}^{\infty} \frac{f(y_{ij})}{\h(y_{ij})}
                    \frac{\hat g(x_{ij})\h(x_{ij})}{\h(x)}
\le \h_*^{-2}\|\h\|_\infty \|f\|_\infty \cdot K'b\e^{m+\a}  \\
=&\h_*^{-2}\|\h\|_\infty C_b\|f\|_\B \cdot  K'b\e^{m}d(x,y)^\a
=C_bK'b\e^{m}_1 \h_*^{-2}\|\h\|_\infty \|f\|_\B d(x,y)^\a,
\end{split}
\end{equation*}
where $C_b$ is given
in Lemma~\ref{LBsp2} which depends on $\e_0$.

Hence the result of part (1) holds
with $D_\H=C_b \h_*^{-1}(2\dd'+K'b\e^{m}_1 \h_*^{-1}\|\h\|_\infty)$.

Part (ii) and (iii) can be proved by using the same estimates with
the same adjustments as in the proof of Lemma~\ref{LLYspV}.
\end{proof}

%The next sublemma means that $h(x)\ge h_*$ for almost every $x\in \wh X$.

\begin{Sublemma}\label{SLhbound}
There is a $\h_*>0$ such that $\h(x)\ge \h_*$ for $\nu$-a.e. $x\in \wh X$.
\end{Sublemma}

\begin{proof}
By Lemma 3.1 in \cite{Ss}, there is a ball $B_\e(z)\subset \wh X$
such that $\disp \Einf_{B_\e(x)}\h \ge \h_-$ for some constant $\h_->0$.
By Assumption~T${}''$(c${}'$), there is $\wt N>0$ such that
$T^{\wt N}B_\e(z)\supset X$.
Then for any $x\in \wt X$, there is $y_0\in B_\e(z)$ such that
$T^{\wt N}y_0=x$. Since $|\det DT|$ is bounded above, we have
$g_*:=\inf\{g(y): \ y\in X\}>0$.  Hence, for $\hn$-almost every $x$,
$$
\h(x)=(\P^{\wt N}\h)(x)
=\sum_{T^{\wt N}y=x} \h(y)\prod_{i=0}^{{\wt N}-1} g(T^iy)
\ge \h(y_0)\prod_{i=0}^{{\wt N}-1} g(T^iy_0)
\ge \h_- g_*^{\wt N}.
$$
The result follows with $\h_*=\h_- g_*^{\wt N}$.
\end{proof}

\begin{Sublemma}\label{SLhHolder}
Let $\e_0$ be as in Lemma~\ref{LLYspV}.
Then there exists $J_{\h}>0$ and $\e_*\in (0, \e_0]$
such that for any $x,y\in V_I$ with $d(x,y)\le \e_*$, $I\in\I$,
$$
\frac{\h(x)}{\h(y)}\le e^{J_{\h} d(x,y)^\a}.
$$
\end{Sublemma}

\begin{proof}
Since $\h$ is the unique fixed point of $\wh \P$, we know that
$\h=\lim_{n\to \infty} {\wh \P}^n \1_{\wh X}$, where the convergence is in $L^1(\hn)$.
Now we consider the sequence $f_n:={\wh \P}^n \1_{\wh X}$.

We will prove that there is $J_{\h}>0$ and $\e_*\in (0,\e_0]$ such that
for any $n\ge 0$ and  for any $x,y\in V_I$, $I\in \I$,  with $d(x,y)\le \e_*$,
\begin{equation}\label{fHolderh}
\frac{f_n(y)}{f_n(x)}\le e^{J_{\h}d(x,y)^\a}.
\end{equation}
%Since $h\in \Q$ and therefore $h$ is bounded above on $\wh X$,
%this implies $|h(x)-h(y)|\le J_{\h} d(x,y)^\a$ for a larger $J_{\h}$.
%By our assumption any ball $B_\e(x)$ satisfies
%$T^{\tilde N(x,\e)}B_\e(x)\supset X$.  It implies taht $h$ is bounded below
%(see the last paragraph of Theorem A in \cite{HV}).
%Hence we get $|h(x)^{-1}-h(y)^{-1}|\le J_{\h} d(x,y)^\a$ if we increase $J_{\h}$
%further.

Clearly \eqref{fHolderh} is true for $n=0$ since $f_0(x)=1$
for any $x$.
We assume that it is true up to $f_{n-1};$ we then   consider $f_n$.

Note that $f_n/\h=(1/\h){\wh \P}^n (h\cdot \1_{\wh X}/\h)={\wh{\L}}^n(\1_{\wh X}/\h)$,
where ${\wh \L}$ is the normalized transfer operator defined by
${\wh \L}(f)=(1/\h){\wh \P}(\h f)$.
%\\{\bf (WARNING. HUYI, where is used this $\hat \L$? Sure we need it?)}\\
Then  there are $f_*\ge \h_*/\h^*$
and $f^*\le \h^*/\h_*$ such that $f_*\le f_{n}(x)\le f^*$
for every $x\in \wh X$ and $n\ge 0$, where $\h^*$ and $\h_*$
are the essential upper and lower bound of $\h$ respectively.
Let also set:
$g_*=\inf_x f_1(x)=\inf_{x}\sum_{j=1}^K \sum_{i=0}^\infty \hat g(x_{ij})$.

Let us set  again  $b=1;$ then put $\dd>0$ as in Assumption T${}''$(e).
Let us take $J_{\h}>2\dd s^\a/(1-s^\a)$ so that we have $(J_{\h}+\dd)s^\a\le J_{\h}(1+s^\a)/2$.
Then we choose $\e_*\in (0, \e_0]$ small enough such that for any
$\e\in [0,\e_*]$,
$$
e^{J_{\h}(1+s^\a )\e^\a/2}+\frac{f^*K'b\e^{m+\a}}{f_*(g_*-K'b\e^{m+\a})}\le e^{J_{\h}\e^\a}.
$$

For any $x,y$ in the same $V_I$ with $d(x,y)=:\e\le \e_*$,
we choose $N=N(\e)$ as in Assumption T${}''$(e).
Let us denote with $[f_n]_N(x) =\sum_{j=1}^K \sum_{i=0}^N \hat g(x_{ij})f_{n-1}(x_{ij})$
and $\{f_n\}_N(x) = f_n(x)-[f_n]_N(x)
=\sum_{j=1}^K \sum_{i=N+1}^\infty \hat g(x_{ij})f_{n-1}(x_{ij})$.
We have
\begin{equation*}%\label{fsum3}
\begin{split}
& \frac{[f_n]_N(y)}{[f_n]_N(x)}
=\frac{\sum_{j=1}^K \sum_{i=0}^N \hat g(y_{ij})f_{n-1}(y_{ij})}
      {\sum_{j=1}^K \sum_{i=0}^N \hat g(x_{ij})f_{n-1}(x_{ij})} \\
\le \sup_{1\le j\le K; 0<i\le N}& e^{\dd d(x_{ij},y_{ij})^\a}e^{J_{\h}d(x_{ij},y_{ij})^\a}
\le e^{(\dd+J_{\h})s^\a d(x,y)^\a}
\le e^{J_{\h}(1+s^\a)\e^\a/2}.
\end{split}
\end{equation*}
We also get
$$
\{f_n\}_N(y)
=\sum_{j=1}^K \sum_{i=N+1}^\infty \hat g(y_{ij})f_{n-1}(y_{ij})
\le f^* \sum_{j=1}^K \sum_{i=N+1}^\infty \hat g(y_{ij})
\le f^* K'be^{m+\a}.
$$
On the other hand,
$$
[f_n]_N(x)
=\sum_{j=1}^K \sum_{i=N+1}^\infty \hat g(y_{ij})f_{n-1}(y_{ij})
\ge f_*\sum_{j=1}^K \sum_{i=1}^N \hat g(y_{ij})
\ge f_* (g_*-K'be^{m+\a}).
$$
By the choice of $\e_*$, we obtain
$$
\frac{f_n(y)}{f_n(x)}
\le \frac{[f_n]_N(y)+\{f_n\}_N(y)}{[f_n]_N(x)}
\le e^{J_{\h}(1+s^\a)\e^\a/2} + \frac{f^* K'b\e^{m+\a}}{f_* (g_*-K'b\e^{m+\a})}
\le e^{J_{\h}\e^\a}.
$$
This  implies \eqref{fHolderh} holds for $n$ since we have set $\e=d(x,y)$.
\end{proof}

\begin{Lemma}\label{LR_n}
There exists a constant $C_R>0$ such that $\|R_n\|_\B\le C_R d_n^{m/(m+\a)}$
for all $n>0$.

If, moreover, $T$ satisfies Assumption~T${}''$(e${}'$),
then $\|R_n\|_\B\le C_R d_n$ for all $n>0$.
\end{Lemma}

\begin{proof}
Since $R_i=\sum_j R_{ij}$, we only need to prove the results
for $R_{ij}$.

Let $s_{ij}(x)$ be the norm of $||D\wh T_{ij}^{-1}(x)||$,
and $s_{ij}=\max\{s_{i,j}(x): x\in B_{\e_0}(Q_0)\}$.
Note that $\{\tau>i\}\subset T^{-1}V$ for all large $i$.
We may suppose that $i$ is sufficiently large so that
$B_{s_{ij}\e_1}(U_{ij})\subset \wh T_{ij}^{-1}V;$ we then take $f\in \B$ with $\|f\|_\B=1$.

By using \eqref{fRijint} and \eqref{fUijb}, we apply arguments similar to
\eqref{fRijL1} and get
\begin{equation}\label{fRijL1norm}
\|R_{ij}f\|_1
=\int_{U_{ij}} |f|d\hn
\le \|f\|_\infty \hn({U_{ij}})
\le C_b \hn(Q_0) d_{ij} \|f\|_\B.
\end{equation}

Next, we consider $|R_{ij}f|_\B$.
%First we
Note that for any $I\in \I$, $f|_{V_I}\in \H^\a(V_{I},H)$
for some $H\le \|f\|_\B$.
So $\osc\bigl(f/\h, \ B_{s\e}(\cdot)\bigr)
\le 2^\a s^\a\e^\a H\le 2^\a s^\a\e^\a\|f\|_\B$.
Moreover Sublemma~\ref{SLhHolder} implies that
$\osc\bigl(\h, \ B_{\e}(x)\bigr)\le 2^\a J_{\h}'\e^\a$ for all $x$ with
$B_{\e}(x)\in V_I$ and with  $J_{\h}'\ge J_{\h}>0$.
By Proposition~3.2(3) in \cite{Ss} we now have:
$$\osc\bigl(f, \ B_{s_{ij}\e}(\cdot)\bigr)\le
\osc\bigl(f/\h, \ B_{s_{ij}\e}(\cdot)\bigr)\h_*
+\osc\bigl(\h, \ B_{s_{ij}\e}(\cdot)\bigr)\|f\|_\infty/\h_*
\le b_1\e^\a\|f\|_\B,
$$
where $b_1=2^\a (H\h_*+J_{\h}'C_bh_*^{-1})s_{ij}^\a$.
By arguments similar to \eqref{fRijint} and \eqref{fRijL2},
\begin{equation}\label{fRijosc}
\begin{split}
&\int \! R_{ij}\osc\bigl(f, \ B_{s_{ij}\e}(\cdot))d\hn
= \int_{U_{ij}}\osc\bigl(f, \ B_{s_{ij}\e}(\cdot))d\hn \\
\le &b_1\e^\a \|f\|_\B \hn(U_{ij})
\le b_1\e^\a d_{ij}\hn(Q_0)\|f\|_\B
\le a_1\e^\a d_{ij}\|f\|_\B,
\end{split}
\end{equation}
where $a_1=b_1\nu(Q_0)$.\footnote{The estimate (\ref{fRijosc}) shows the difference with the analogous bound (\ref{fRijL2}) and justifies the introduction of the new Banach space. In fact we can now use the local H\"older property for $f$ to get an upper bound of the integral of the oscillation simultaneously in terms of the volume of $U_{ij}$, of $\epsilon$ and of the norm of $f$. The change of variable sending $U_{ij}$ to $Q_0,$ will finally produce the determinant $d_{ij}$ which will give a better upper bound for $||R_n||.$}  Also,
$$
\hn\bigl(\wh T_{ij}^{-1}B_\e(\partial \wh TU_{ij}) \bigr)
=\int_{B_\e(\partial \wh TU_{ij})} {\hat g}d\hn
\le d_{ij}\cdot \hn\bigl(B_\e(\partial U_{0})\bigr)
\le d_{ij}\cdot b_2\e,
$$
for some $b_2>0$ independent of $\e$.
Hence,
\begin{equation}\label{fRijdisc}
G_{ij}(x, \e,\e_0)=2d_{ij}\cdot b_2\e/\hn(B_{(1-s)\e_0}(x))
\le a_2 d_{ij}\e,
\end{equation}
where $a_2=2b_2/\hn(B_{(1-s)\e_0}(x))$.
Note that $\int \osc(f, B_{\e_0}(x_{ij})) d\hn\le \e_0^\a |f|_\Q$,
and $\|f\|_1+\e_0^\a|f|_\Q\le \|f\|_\Q\le \|f\|_\B$.
Therefore for any $\e\in (0, \e_0]$ and $i<N(\e)$ and by using \eqref{fRija},
\eqref{fRijosc}, \eqref{fRijL1norm} and \eqref{fRijdisc} we get
\begin{equation}
\begin{split}\label{fRijQ}
 |R_{ij}f|_\Q
\le &\bigl[(1+\dd \e^\a) a_1 + 2\dd C_b \nu(Q_0)+
  2(1+\dd \e^\a) a_2\e^{1-\a} \bigr]d_{ij}\|f\|_\B\\
\le & C_2' d_{ij}\|f\|_\B,
\end{split}
\end{equation}
where
$C_2'=(1+\dd \e^\a) a_1 + 2\dd C_b \nu(Q_0)+ 2(1+\dd \e^\a) a_2\e^{1-\a}$.

For $\e\in (0, \e_0],$  $i>N(\e)$ and
by Assumption T${}''$(e) we have $d_{ij}\le b\e^{m+\a}$.
Hence, $\e^{-a}\le (b^{-1}d_{ij})^{-\a/(m+\a)}$.
Hence by \eqref{fRijb}, we have
\begin{equation}
\begin{split}\label{fRij2}
 |R_{ij}f|_\Q
\le &2(\c_m\e_0^m)^{-1}\cdot \|f\|_\Q\cdot \e^{-\a}\cdot d_{ij}  \\
\le &2(\c_m\e_0^m)^{-1} b^{\a/(m+\a)}d_{ij}^{1-\a/(m+\a)}\|f\|_\Q
\le C_2''d_{ij}^{m/m+\a}\|f\|_\B,
\end{split}
\end{equation}
where $C_2''=2(\c_m\e_0^m)^{-1} b^{\a/(m+\a)}$.
Therefore  we get that $|R_{ij}f|_\Q\le C_2d_i^{m/m+\a}$,
where $C_2= \max\{C_2', C_2''\}$.

Now we consider $|R_{ij}f|_\H$.
As in the proof of Lemma~\ref{LLYspH}, for any $x,y\in U_{ij}$,
\begin{equation}\label{fRijH1}
\begin{split}
\Bigl|\frac{{R_{ij}}f(x)}{\h(x)}-\frac{{R_{ij}}f(y)}{\h(y)}\Bigr|
\le &\Bigl|\frac{\hat g(x_{ij})f(x_{ij})}{\h(x)}
          -\frac{\hat g(y_{ij})f(y_{ij})}{\h(y)}\Bigr|\\
=&\frac{\hat g(x_{ij})\h(x_{ij})}{\h(x)}
\Bigl|\frac{f(x_{ij})}{\h(x_{ij})}-\frac{f(y_{ij})}{\h(y_{ij})}\Bigr| \\
+&\frac{|f(y_{ij})|}{\h(y_{ij})}
\Bigl|\frac{\hat g(x_{ij})\h(x_{ij})}{\h(x)}
    - \frac{\hat g(y_{ij})\h(y_{ij})}{\h(y)} \Bigr|. \\
\end{split}
\end{equation}
Note that
$\Bigl|{f(x_{ij})}/{\h(x_{ij})}-{f(y_{ij})}/{\h(y_{ij})}\Bigr|
\le |f|_\H d(x_{ij},y_{ij})^\a \le \|f\|_\B s_{ij}^\a d(x,y)^\a$
and $\hat g(x_{ij})\h(x_{ij})/{\h(x)}\le (\h^*/\h_*)d_{ij}$.
Then the first term in the right hand side of \eqref{fRijH1} is bounded by
$a_3d_{ij} \|f\|_\B d(x,y)^\a$, where $a_3=(\h^*/\h_*)s_{ij}^\a$.

Let us take $\e=d(x,y)$; if $i\le N(\e)$, then by \eqref{fRijL1norm0},
$$|{\hat g(x_{ij})\h(x_{ij})}/{\h(x)} - {\hat g(y_{ij})\h(y_{ij})}/{\h(y)}|
\le 2\dd'(\h^*/\h_*)d_{ij}d(x,y)^\a.
$$
Since $f(y_{ij})/\h(y_{ij})\le \|f\|_\infty/\h_*\le C_b \h_*^{-1}\|f\|_\B$,
the last term in \eqref{fRijH1} is bounded by
$a_4d_{ij}\|f\|_\B d(x,y)^\a$, where $a_4=2C_bJ'(\h^*/\h_*^2)$.
Therefore we obtain $|R_{ij}f|_\H\le C_3'd_{ij}\|f\|_\B$, where $C_3'=b_1+b_2$.

If $i\ge N(\e)$, then by the first inequality of \eqref{fRijH1},
the left side of the inequality is bounded by
$\max\{{\hat g(x_{ij})f(x_{ij})}/{\h(x)},  {\hat g(y_{ij})\h(y_{ij})}/{\h(y)}\}
\le d_{ij} \|f\|_\infty /\h_*$.
By the same arguments as for \eqref{fRij2} we  get that
$$
|R_{ij}f|_{\H} \!
\le\! \e^{-\a}d_{ij} \|f\|_\infty /\h_*
\le C_b \h_*^{-1}b^{\a/(m+\a)}d_{ij}^{m/(m+\a)}\|f\|_\B \!
=C_3''d_{ij}^{m/(m+\a)}\|f\|_\B,
$$
where $C_3''=C_b \h_*^{-1}b^{\a/(m+\a)}\|f\|_\B$.  Then we conclude that
$|R_{ij}f|_{\H} \le C_3d_{ij}^{m/(m+\a)}\|f\|_\B$, where $C_3=\max\{C_3',C_3''\}$.

The conclusion of the first part follows by setting
$C_R=C_1+C_2+C_3$.

If $T$ satisfies Assumption~T${}''$(e${}'$),
then we can regard $N(\e)=\infty$ for any $\e>0$.
Hence we  obtain $\|R_{ij}f\|_\B\le C_R d_{ij}\|f\|_\B$ with $C_R=C_1+C_2+C_3'$.
\end{proof}

\section*{Acknowledgments}
We wish to thank Ian Melbourne for his interest in our maps, for his helpful comments and advices. \\We would like to thank also Romain Aimino who was trapped in interminable discussions on aperiodicity and function spaces. \\  SV was supported by
the project APEX Syst\`emes dynamiques: Probabilit\'es et Approximation Diophantienne
PAD funded by the R\'eegion PACA (France). We finally thank the anonymous referee for a very careful reading of the paper and whose comments and
suggestions helped us to improve our article.

%%%%%%%%%%%%%%%%%%%%%%%%%%%%%%%%%%%%%%%%%%%%%%%%%%%%%%%%%%%%%%%%%%%%%%
%%%%%%%%%%%  REFERENCES
%%%%%%%%%%%%%%%%%%%%%%%%%%%%%%%%%%%%%%%%%%%%%%%%%%%%%%%%%%%%%%%%%%%%%%


\begin{thebibliography}{PGB}

\bibitem[AD]{AD} J. Aaronson and M. Denker,
Local limit theorems for partial sums of stationary sequences generated by Gibbs-Markov maps,
{\it Stoch. \& Dynam.}, {\bf 1} (2001), 193--237

\bibitem[ADSZ]{ADSZ} J. Aaronson, M. Denker, O. Sarig and R. Zweim\"uller,
Aperiodicity of cocycles and conditional local limit theorems,
{\it Stoch. \& Dynam.}, {\bf 4} (2004), 31--62

%\bibitem[ABV]{ABV} J.F. Alves, C. Bonatti and M. Viana,
%SRB measures for partially hyperbolic systems whose central direction
%in mostly expanding,
%{\it Invent. Math.}, {\bf 140}, (2000), 351--398

\bibitem[BG]{BG} A. Boyarsky and P. G\'ora,
{\it Laws of Chaos : Invariant Measures and Dynamical Systems in One
Dimension}, Probability and its Applications, Birkhauser, 1997

\bibitem[Bo]{Bo} R. Bowen
{\it Equilibrium states and the ergodic theory of Anosov diffeomorphisms},
Lecture Notes in Math. 470,  Springer,   New York, 1975

\bibitem[Br]{Br} A. Broise,
{\it Transformations dilatantes de l'intervalle et th\'eor\`emes limites},
5-110, Ast\'erisque, {\bf 238} (1996)

\bibitem[Go]{Go} S. Gou\"ezel,
Sharp polynomial estimates for the decay of correlations,
{\it Israel J. Math.}, {\bf 139} (2004), 29--65

\bibitem[He]{He} H. Hennion,
Sur un th\'eor\`eme spectral et son application aux noyaux lipchitziens,
{\it Proc. Amer. Math. Soc.}, {\bf 118} (1993), 627--634

\bibitem[HH]{HH} H. Hennion and L. Herv\'e,
{\it Limit theorems for Markov chains and Stochastic Properties
of Dynamical Systems by Quasicompactness, Lect. Notes Math.}, 1766,
Springer-Verlag, 2001

\bibitem[Hu]{Hu} H. Hu,
Decay of correlations for piecewise smooth maps
with indifferent fixed points,
{\it Ergodic Theory Dynam. Systems}, {\bf 24} (2004), 495--524

\bibitem[HPT]{HPT} H. Hu, Ya. Pesin and A. Talitskaya,
A Volume Preserving Diffeomorphism with Essential Coexistence
of Zero and Nonzero Lyapunov Exponents
{\it Comm. Math. Phys.}  {\bf 319}  (2013),  331--378


\bibitem[HV]{HV} H. Hu and S. Vaienti,
Absolutely Continuous Invariant Measures for Nonuniformly Expanding Maps,
{\it Ergodic Theory Dynam. Systems}, {\bf 29} (2009), 1185--1215

\bibitem[IM]{IM} C.T. Ionescu Tulcea and G. Marinescu,
Th\'eorie ergodique pour des classes d'op\'erations non compl\`etement continues
(French),
{\it Ann. of Math.}, {\bf 52} (1950), 140--147

\bibitem[Kk]{Kk} S. Kakutani,
Induced measure preserving transformations,
{\it Proc. Imp. Acad. Tokyo}, {\bf 19} (1943), 635--641

%\bibitem[KS]{KS} S. Kochen and C. Stone,
%A note on the Borel-Cantelli lemma,
%{\it Ill. J. Math.}, {\bf 8} (1964), 248--251

\bibitem[Kr]{Kr} U. Krengel,
{\it Ergodic theorems},
de Gruyter Studies in Mathematics, 6. Walter de Gruyter \& Co., Berlin, 1985.


\bibitem[LY]{LY} A. Lasota and J. Yorke,
On the existence of invariant measures for piecewise monotonic transformations,
{\it Trans. Amer. Math. Soc.}, {\bf 186}, (1973) 481--488

\bibitem[LSV]{LSV} C. Liverani, B.Saussol and S. Vaienti,
A probabilistic approach to intermittency,
{\it Ergodic Theory Dynam. Systems}, {\bf 19} (1999), 671--685

\bibitem [MT] {MT} I. Melbourne, D. Terhesiu, Decay of correlations for nonuniformly expanding systems with general return times, {\it Ergodic Theory Dyn. Syst.},  {\bf 34}, (2014) 893-918

\bibitem[PP]{PP} W. Parry and M Pollicott,
Zeta functions and the periodic orbit structure of hyperbolic dynamics,
{\it Ast\'erisque}, {\bf 187-188} (1990)

%\bibitem[Ps]{Ps} Ya. Pesin,
%Characteristic Ljapunov exponents, and smooth ergodic theory (Russian),
%{\it Uspehi Mat. Nauk}, {\bf 32} (1977), 55--112


\bibitem[PY]{PY} M. Pollicott and M. Yuri,
Statistical Properties of maps with indifferent periodic points,
{\it Comm. Math. Phys.}, {\bf 217} (2001), 503--520

\bibitem[Qu]{Qu} A. Quas,
Non-ergodicity for $C^1$ expanding maps and $g$-measures (English summary),
{\it Ergodic Theory Dynam. Systems}, {\bf 16} (1996), 531--543

\bibitem[Sr]{Sr} O. Sarig,
Subexponential decay of correlations,
{\it Invent. Math.}, {\bf 150} (2002), 629--653

\bibitem[Ss]{Ss} B. Saussol,
Absolutely continuous invariant measures for multidimensional expanding maps,
{\it Israel J. Math.}, {\bf 116} (2000), 223--248

\bibitem[Ya]{Ya} J.-A. Yan,
A Simple Proof of Two Generalized Borel-Cantelli Lemmas
in {\it Lecture Notes in Mathematics}, 1874,  Springer-Verlag, 2006,  77--79

%\bibitem[Yo1]{Yo1}  L.-S. Young,
%Statistical properties of dynamical systems with some hyperbolicity,
%{\it Ann. of Math.}, {\bf 147} (1950),  585--650

\bibitem[Yo2]{Yo2}  L.-S. Young,
Recurrence times and rates of mixing,
{\it Israel J. Math.}, {\bf 110} (1999), 153--188



\bibitem[Zm]{Zm} W. Zeimer,
{\it Weakly Differentiable Functions},
Graduate Text in Mathematics, 120, Springer (1995)

\bibitem[Z1]{Z1} R. Zweim\"uller,
Ergodic structure and invariant densities of non-Markovian interval maps
with indifferent fixed points,
{\it Nonlinearity}, {\bf 11} (1998), 1263--1276

\bibitem[Z2]{Z2} R. Zweim\"uller,
Ergodic properties of infinite measure-preserving interval maps
with indifferent fixed points,
{\it Ergodic Theory Dynam. Systems}, {\bf 20} (2000), 1519--1549



\end{thebibliography}
\end{document}